\documentclass[12pt, reqno, oneside]{amsart}
\usepackage{macros}
\usepackage[
backend=biber,
style = numeric,
url = false,
doi = false,
isbn = false
]{biblatex}
\appto{\bibsetup}{\raggedright}
\renewbibmacro{in:}{}
\numberwithin{equation}{section}
\title{Nodal degeneration of chiral algebras II:\@ Local structure and chiral Zhu algebras} 
\author{Elchanan Nafcha} 
\date{}


\begin{document}

\pagestyle{plain}

\begin{abstract}
    Given a universal factorization algebra $\cA$, we constructed in our previous paper a derived associative algebra $\fZ_{\cA}^0$, together with chiral modules $\fZ_{\cA}^+$ and $\fZ_{\cA}^-$ associated to a puncture, and a chiral bimodule $\fZ_{\cA}$ associated to a node. Furthermore, these constructions assemble to a factorization $\cA$-module over any family of nodal punctured curves. In this paper, we show that in the case where $\cA$ is constructed from a quasi-conformal vertex algebra $V$, the zeroth homology algebra $H^0\fZ_{\cA}$ is naturally isomorphic to Zhu's associative algebra $A(V)$, and we identify $H^0\fZ_{\cA}$ with the bimodule underlying the mode-transition algebra of Damiolini-Gibney-Krashen. We also give an explicit description of the smoothing module $H^0\tilde{\fZ}_{\cA}$ which describes the deformation of $H^0\fZ_{\cA}$ which we attach to a smoothing family of a nodal curve. We therefore get a geometric interpretation of the Zhu algebra and the mode-transition algebra, as the integration of a factorization algebra over a certain compactification of configuration spaces of punctured nodal curves. 
\end{abstract}

\maketitle
\tableofcontents

\section{Introduction}

The theory of vertex algebras provides a mathematical formalism for the structure of local observables in a two-dimensional conformal field theory. It was introduced by Borcherds~\cite{Bor}, Frenkel, Lepowsky, and Meurman~\cite{FLM} in their work on the representation theory of the monster group, and has important applications to the study of infinite dimensional Lie algebras~\cite{FZ} and moduli stacks of $G$-bundles~\cite{LS}. Its basic structure is given by a vector space $V$ of states, together with a ``state-field correspondence'' \[Y : V \to \operatorname{End}V[\![z^{\pm 1}]\!]\] assigning to a state $v\in V$ a field of operators $v(z)$ over the formal disk $\hat{D}$, see~\ref{subsec:VOA} for a complete definition. A more general state space is given by the structure of a $V$-module, which roughly amounts to a vector space $M$ and an assignment of fields of operators over $\hat{D}$ for each $v\in V$: \[Y^M : V \to \operatorname{End}M[\![z^{\pm 1}]\!]~.\] 

A derived and global generalization to the theory of vertex algebras is provided by Beilinson and Drinfeld's chiral and factorization algebras~\cite{BD}. A factorization algebra over a smooth complex curve $X$ is, roughly speaking, an assignment of a vector space (or a chain complex) $\cA_{x_1,\dots,x_n}$ to any finite subset $\{x_1,\dots,x_n\} \subset X$, as well as isomorphisms \[\cA_{x_1,\dots,x_n} \simeq \cA_{x_1} \otimes \dotsb \otimes \cA_{x_n}\] whenever $x_i \cap x_j = \emptyset$ for $i \neq j$. Any $\operatorname{Aut}\hat{D}$-equivariant vertex algebra gives rise to a factorization algebra, such that $\cA_{x} \simeq V$ for each $x$, see~\cite{FBZ}. A factorization algebra constructed this way is known as a universal factorization algebra --- a factorization algebra defined over any smooth curve in a compatible way. Factorization algebras provide a crucial ingredient for the proof of the geometric Langlands conjecture~\cite{arinkinProofGeometricLanglands2024}. Intuitively, they allow one to decompose global objects over a curve into local pieces around finite subsets, and then reassemble them into a global object using factorization homology, given by integrating over all finite subsets \[\int_X \cA \simeq \colim_{x_1,\dots,x_n \in X} \cA_{x_1,\dots,x_n}\] (see Definition~\ref{def:fact-homology}). Its zeroth homology is what is known as the dual space of conformal blocks in the theory of vertex algebras.

An important tool in understanding the category of $V$-modules is the Zhu algebra $A(V)$ defined in~\cite{Zhu}. This algebra is generally much easier to understand than $V$ itself, and can be viewed as an approximation to the representation theory of $V$. Indeed, Zhu constructed a map \begin{equation} \label{eq:zhu-to-V}
    A(V)\mh\Mod \to V\mh\Mod
\end{equation} and, by taking a maximal quotient, a bijection on irreducible modules. 

As has already been noted by Zhu~\cite{Zhu}, the algebra $A(V)$ seems to control the behavior of a vertex algebra along nodal degenerations, and has an important role in the gluing formula, relating $V$-conformal blocks over a nodal curve to $V$-conformal blocks over its normalization. Indeed, this algebra plays an important role in the proof of the Verlinde formula for rational $C_2$-cofinite vertex algebras~\cite{DGT}. An extension of this algebra, the mode-transition algebra $\fA(V)$, is used in~\cite{DGK} to extend any vertex algebra to stable curves. In~\cite{vanekerenFirstChiralHomology2024}, Van Ekeren and Heluani prove that the analytic limit of the zeroth and first factorization homology groups along nodal degeneration of elliptic curves is equivalent to the zeroth and first Hochschild homology groups of the Zhu algebra. Intuitively, the Zhu algebra plays the role of $SO(2)$-invariants of factorization homology over an infinite length cylinder~\cite{BZMO}, so that the Verlinde formula gives an algebraic limit of gluing factorization algebras over Riemann surfaces along a cylinder in the topological~\cite{AF} and analytic~\cite{bartelsConformalNetsII2017} contexts. However, to the best of our knowledge, no general formulation of these ideas appears in the literature.  

In part I~\cite{Naf2}, we extended a universal factorization algebra $\cA$ to any punctured and nodal curve $(X, x_1,\dots,x_n)$, so that its value over each puncture $x_i$ is a certain ``right'' chiral module $\fZ_{\cA}^+$, or an isomorphic ``left'' chiral module $\fZ_{\cA}^-$, and its value over each node is a certain chiral bimodule $\fZ_{\cA}$. These are related by tensoring over a certain derived associative algebra $\fZ_{\cA}^0$ which provides the basis for our gluing formula for chiral homology, and which we refer to $\fZ_{\cA}^0$ as the chiral Zhu algebra associated to $\cA$. Namely, we define a left $\fZ_{\cA}^0$-module structure on $\fZ_{\cA}^+$, and a right $\fZ_{\cA}^0$-module structure on $\fZ_{\cA}^-$ so that we have an isomorphism $\fZ_{\cA} \simeq \fZ_{\cA}^+ \underset{\fZ_{\cA}^0}{\otimes} \fZ_{\cA}^-$. Our main goal in this paper is to relate these constructions to familiar constructions in the theory of vertex algebras. More precisely, we prove:

\begin{theorem}
    (see Theorem~\ref{thm:H0ZA}) For a universal factorization algebra $\cA$ corresponding to a quasi-conformal vertex algebra $V$, there is an isomorphism of associative algebras \[H_0\fZ_{\cA}^0 \simeq A(V)\] between the zeroth homology of the chiral Zhu algebra associated to $\cA$ and the usual Zhu algebra associated to $V$. We further identify $H_0\fZ_{\cA}^+$ and $H_0\fZ_{\cA}^-$ with the vertex module induced by $A(V)$ viewed as a module over itself. Consequently, we obtain an isomorphism \[H_0\fZ_{\cA} \simeq \fA(V)\] between the node bimodule $H^0\fZ_{\cA}$ and the vertex bimodule underlying the mode-transition algebra $\fA(V)$ associated to $V$. 
\end{theorem}

Furthermore, by taking a smoothing family of a nodal curve, we obtain a deformation $H_0\tilde{\fZ}_{\cA}$ of the mode transition algebra $\fA(V)$, which we describe explicitly as a quotient \[H_0\tilde{\fZ}_{\cA} \simeq \fU(V)/\fI(V)\] of a deformation of the universal enveloping algebra $U(V)$ of $V$.

The chiral Zhu algebra $\fZ_{\cA}^0$ plays a central role in the gluing formula for our extension of universal factorization algebras to nodal and punctured curves. Namely, for a universal factorization algebra $\cA$ one can define a sheaf over the moduli stack $\cM_{g}$ of smooth genus-$g$ curves, whose fiber over a smooth curve $X$ is given by the factorization homology complex $\int_X \cA$. In~\cite{Naf2} we extend this sheaf to the DM stacks $\bcM_{g,n}$. The factorization homology of a universal factorization algebra $\cA$ over a punctured curve $X \backslash \{x\}$ may be interpreted as the factorization homology of the associated factorization algebra $\cA_X$ over $X$ with coefficients in the chiral module $\fZ_{\cA}^+ \simeq \fZ_{\cA}^-$: \[\int_{X \backslash \{x\}} \cA \simeq \int_{(X;x)} (\cA_X,\fZ_{\cA}^+) \simeq \int_{(X;x)} (\cA_X,\fZ_{\cA}^-)~.\] Similarly, the factorization homology over a nodal curve $X$ with a node at $p \in X$ can be interpreted as factorization homology of the restriction to the smooth locus $\jmath : X \backslash \{p\} \into X$ with coefficients in the chiral bimodule $\fZ_{\cA}$: \[\int_X \cA \simeq \int_{(X;p)} (\jmath_*\cA_{X \backslash \{p\}},\fZ_{\cA})~.\] 

Using the $\fZ_{\cA}^0$-module structure on $\fZ_{\cA}^+ \simeq \fZ_{\cA}^-$, we obtain an induced module structure on factorization homology. We then have the following gluing formula: \[\int_{X \underset{x \sim y}{\cup} Y} \cA \simeq \int_{X \backslash \{x\}} \cA \underset{\fZ_{\cA}^0}{\otimes} \int_{Y \backslash \{y\}} \cA~.\] Thus, one may think of $\fZ_{\cA}^0$ as the algebraic model for factorization homology over a circle, in analogy with the gluing formula for topological factorization homology. 

Finally, given a family of curves $X/\hat{D}$ with a single node $p \in X_0$ over $0 \in \hat{D}$ such that the formal completion $X_x^{\land}$ is isomorphic to $\Spf k[\![x,y,t]\!]/(xy-t) \to \Spf k[\![t]\!] = \hat{D}$, we obtain a $\hat{D}$-family of factorization modules $\tilde{\fZ}_{\cA}$, which gives an $\hat{\AA}^1$ deformation of $\fZ_{\cA}$. Trivializing this deformation will then amount to a sewing formula, namely identifying the deformation of factorization homology over a smoothing family with factorization homology over the trivial deformation of the normalization. We describe the deformation $H_0\tilde{\fZ}_{\cA}$ in terms of a quotient of the trivial deformation of the universal enveloping algebra $U(V)$.

In the case where $\cA = \pmb{1}$ is the unit factorization algebra, corresponding to the vertex operator algebra $V = k$ given by the ground field, or equivalently the factorization algebra over any curve given by the dualizing sheaf, we prove vanishing of higher homologies, so that we have \[\fZ_{\pmb{1}}^0 \simeq k;\quad \tilde{\fZ}_{\pmb{1}} \simeq k[\![t]\!]~.\] (see Proposition~\ref{prop:Z-unit}). In particular, we obtain the expected result for its factorization homology: \[\int_{X \backslash {\overline{x}}} \pmb{1} \simeq k\] (see Corollary~\ref{cor:int-unit}). In general, we do not know whether higher homologies of $\fZ_{\cA}^0$ vanish.

The construction of $\fZ_{\cA}^0$ is given by integrating $\cA$ over the moduli $\bcM_{0,2,\Ran}$ of rational chains together with finite stable subsets. Namely, curves of the form \[\PP^1 \underset{\infty \sim 0}{\cup} \PP^1 \cup \dotsb \underset{\infty \sim 0}{\cup} \PP^1\] together with a finite  subset supported away from the points $0,\infty$ of each copy of $\PP^1$, and such that each component contains at least one marked point. A family of such curves may include degenerations of a given component into multiple components. The moduli of such curves can be identified with the stack $\fM_{0,2}^{\sst}$ of two-pointed semistable curves of genus zero. The isomorphism group of each rational chain with $\ell$ irreducible components is given by $\GG_m^{\ell}$, acting by rotation on each component. Thus, we may view the stack $\fM_{0,2}^{\sst}$ as the moduli of rotation-equivariant cylinders. It admits a non-unital associative multiplication given by concatenation of chains, which gives an algebraic model for the associative embedding of disjoint union of cylinders inside a cylinder. The algebra $\fZ_{\cA}^0$ may be described informally as the colimit \[\fZ_{\cA}^0 = \underset{(X,x_0,x_{\infty}) \in \fM_{0,2}^{\sst}, \;y_1,\dots,y_n \in \mathring{X}}{\colim} \cA_{y_1,\dots,y_n}\] where $\mathring{X} = X^{\sm} \backslash \{x_0,x_\infty\}$ is the smooth and punctured locus of $(X,x_0,x_{\infty})$. The multiplication on $\fZ_{\cA}^0$ is then induced from the multiplication on $\fM_{0,2}^{\sst}$, and we can view $\fZ_{\cA}^0$ as the rotation-invariant factorization homology over a cylinder with $\ell \to \infty$ irreducible components. See also~\cite{drummond-coleHomotopicallyTrivializingCircle2014,oanceaDeligneMumfordOperadTrivialization2026} for the relation between rotation invariant topological factorization algebras and the Deligne-Mumford-Knudsen compactification.

More precisely, from the data of a universal factorization algebra $\cA$ we can construct a sheaf over any finite collection of smooth points inside any curve. In particular, we get a sheaf $\Vac(\cA)_{0,2}^{\epsilon}$ over $\bcM_{0,2,\Ran}$, whose fiber over $(X,x_0,x_{\infty},y_1,\dots,y_n)$ is given by $\cA_{y_1,\dots,y_n}$. Furthermore, for a fixed pointed curve $(X,x_0,x_{\infty})$ we have a connection along the space of stable subsets $\{y_1,\dots,y_n\} \subset \mathring{X}$. We then define $\fZ_{\cA}^0$ as the complex of compactly supported sections of $\Vac(\cA)_{0,2}^{\epsilon}$ over $\bcM_{0,2,\Ran}$, which are flat along the fiber of each pointed curve. While each $\mathring{X}$ is not proper, it turns out that the collection of stable configurations of a given cardinality $n$ is a smooth and proper variety, known as the $n$-th Losev--Manin space~\cite{LM}, and is given by a sequence of blow-ups of $\PP^{n-1}$, see~\cite[Section~6]{Has}. In particular, we have a well-defined compactly supported cohomology functor.

\subsection{Notations and conventions}

Throughout this paper, $k$ refers to a fixed algebraically closed field of characteristic zero. We use the language of $(\infty,1)$-categories implicitly, and refer to them as categories. For example, $\Vect_k$ will refer to the derived category of $k$, and $\QCoh(X)$ to the derived category of $X$. We use the theory of ind-coherent sheaves over locally almost of finite type prestacks as developed in~\cite{Gai-ICS}. We use the following notations:

\begin{itemize}
    \item $\Delta$ refers to the category of finite ordered sets and order preserving morphisms. The object $[n]$ refers to the ordered set $(0 < \dotsb < n)$. $\Delta^{\surj} \subset \Delta$ is the subcategory of all objects and surjective morphisms.
    \item $\fSet$ is the category of finite sets and arbitrary maps between them. The subcategory $\fSet^{\surj} \subset \fSet$ contains all objects and surjective morphisms.
    \item $\Cat_k$ is the category of $\Vect_k$-modules inside the category of stable and presentable $\infty$-categories, and exact and continuous maps between them.
    \item For a prestack $X$, we denote $X_{\dR}(R) = X(R^{\red})$ its de Rham stack. For a map of prestacks $X \to S$, we denote $X_{S\mh\dR} \coloneqq X_{\dR} \times_{S_{\dR}} S$ the relative de Rham stack.
    \item $\hat{\AA}^1$ denotes the formal completion of the affine line at the origin $0 \in \AA^1$.
    \item For a two-pointed semistable curve $(X,x_0,x_{\infty}) \in \fM_{0,2}^{\sst}$ and a pointed curve $(Y,y)$, we denote their join by \[(X,x_0,x_{\infty}) \lor (Y,y) = (X \underset{x_{\infty}\sim y}{\cup}Y, x_0); \quad (Y,y) \lor (X,x_0,x_{\infty}) = (Y \underset{y \sim x_{0}}{\cup}X, x_{\infty})~.\]
    \item By a family of curves $X/S$ we always mean a flat and finitely presented morphism $X \to S$ of relative dimension $\leq 1$.
    \item For a curve $X$, we denote by $\Ran X$ the classifying prestack for finite subsets of its de Rham stack $X_{\dR}$. For a family of curves $X/S$, we denote by $\Ran(X/S)$ the classifying $S$-prestack for finite sets of sections of $X_{S\mh\dR} \to S$ (Definition~\ref{def:RanXS}).
    \item For a finite set $I$, $\MDisk_I$ is the moduli stack of $I$-pointed multidisks (see Definition~\ref{def:MDisk}).
    \item For integers $g,n \geq 0$, $\fM_{g,n}^{\sst}$ denotes the stack of $n$-pointed semistable curves of genus $g$, and $\bcM_{g,n}$ the stack of stable $n$-pointed genus $g$ curves.
    \item The stack $\fM_{0,2+1}^{\sst,\land}$ (resp. $\fM_{0,1+2}^{\sst,\land}$, resp. $\fM_{0,2+2}^{\sst,\land}$) denotes the deformation of $\fM_{0,2}^{\sst}$ where we allow another formal deformation of a node at the first (resp. second, resp. both) marked point, see Definitions~\ref{def:punct-modif-stack} and~\ref{def:node-modif-stack}. 
    \item The prestack $\bcM_{0,2,\Ran}$ classifies stable configurations of weight-$\epsilon$ marked points over $\fM_{0,2}^{\sst}$. The prestacks $\bcM_{0,2+1,\Ran}^{\land}$, $\bcM_{0,1+2,\Ran}^{\land}$, and $\bcM_{0,2+2,\Ran}^{\land}$ are given by its formal deformations allowing a formal node at the first, second, and both marked points respectively (see Definition~\ref{def:M02plus1Ran}).
    \item For an integer $\ell > 0$, we denote by $\fA[\ell]$ the quotient stack $\AA^{\ell}/\GG_m^{\ell+1}$ (see Definition~\ref{def:A[l]}).
    \item Given a universal factorization algebra $\cA$, we write $\fZ_{\cA}^0$ for the chiral Zhu algebra, $\fZ_{\cA}$ for its node module, $\fZ_{\cA}^+$ and $\fZ_{\cA}^-$ for its right and left puncture modules, and $\tilde{\fZ}_{\cA}$ for its smoothing module (see Definition~\ref{def:smoothing-module}).
\end{itemize}

\subsection{Overview}

In Section~\ref{sec:VA-FA}, we recall the definitions of (families of) factorization algebras, vertex algebras, and universal factorization algebras. We review the definition of the Zhu associative algebra $A(V)$ and mode transition algebra $\fA(V)$ associated to a vertex algebra $V$, both constructed from the universal enveloping algebra $U(V)$. 

In Section~\ref{sec:ZA} we describe the construction of the chiral Zhu algebra $\fZ_{\cA}^0$ associated to a universal factorization algebra $\cA$, as well as the puncture $\cA$-modules $\fZ_{\cA}^+ \simeq \fZ_{\cA}^-$, their nodal gluing, given by the  node module $\fZ_{\cA}$, and its deformation $\tilde{\fZ}_{\cA}$, the smoothing module. This is done using integration over various spaces of semistable modifications, which we review in~\ref{sec:ss-modif}. We then compute the above objects in the case of the unit factorization algebra, given by the dualizing sheaf.

In Section~\ref{sec:ss}, we construct a spectral sequence for computing $\tilde{\fZ}_{\cA}$ and $\fZ_{\cA}^0$, using a filtration on the Ran space of node-modifications $\bcM_{0,2+2,\Ran}^{\land}$ given by a combination of the cardinality and number of components filtrations. Using this spectral sequence, we identify $H_0\fZ_{\cA}^0$ with $A(V)$, and $H_0\fZ_{\cA}$ with $\fA(V)$, for any quasi-conformal vertex algebra $V$.

\subsection{Acknowledgements}

I am deeply grateful to my advisor, John Francis, for his constant support, guidance, and insight throughout the writing of this paper.

I would like to thank Chiara Damiolini for explaining the mode-transition algebra and its role in the Verlinde formula, and Alexander Beilinson for his thoughtful questions, corrections, and valuable comments.

I am also grateful to André Henriques, Riemundo Heluani, Owen Gwilliam, Dmitry Tamarkin, Yuchen Liu, and Aleksander Karapetyan for many helpful discussions and for generously sharing their ideas with me.

\section{Vertex and factorization algebras} \label{sec:VA-FA}

\subsection{Factorization structures}

In this section we recall the definition of factorization algebras and modules over smooth families of curves, following~\cite{BD,FG}.

\begin{definition} \label{def:RanXS}
    Given a family of smooth curves $X$ over an algebraic stack $S$, define the relative Ran space to be the $S$-prestack \[\Ran(X/S) : (\Spec R \to S) \mapsto \{\text{finite nonempty subsets of }\hom_{S}(\Spec R^{\red}, X)\}~.\] 
\end{definition}

Define the disjoint locus \[\jmath : (\Ran(X/S) \times_S \Ran(X/S))_{\disj} \subset \Ran(X/S) \times_S \Ran(X/S)\] to be the open sub-prestack classifying pairs of subsets with disjoint graphs. We have a union map \[\cup : \Ran(X/S) \times_S \Ran(X/S) \to \Ran(X/S)\] which restricts to a disjoint union map \[\sqcup : (\Ran(X/S) \times_S \Ran(X/S))_{\disj} \to \Ran(X/S)~.\] The union map defines a (non-unital) symmetric monoidal structure in $S$-prestacks \[(\Ran(X/S),\cup) \in \Comm(\PreSt_S, \times_S)~,\] while the disjoint union defines a (non-unital) symmetric monoidal structure in the correspondence category of $S$-prestacks \[(\Ran(X/S),(\jmath,\sqcup)) \in \Comm(\PreSt_S^{\corr},\times_S)~,\] with multiplication given by the correspondence \[\Ran(X/S) \times_S \Ran(X/S) \xleftarrow{\jmath} (\Ran(X/S) \times_S \Ran(X/S))_{\disj} \xrightarrow{\sqcup} \Ran(X/S)~.\] Assuming $S$ is of finite type, all prestacks involved are locally almost of finite type (laft). We can therefore use the lax symmetric monoidal functor \[\IndCoh : \PreSt_S^{\operatorname{laft},\corr} \to \Cat_S = \IndCoh(S)\mh\Mod(\Cat_k)\] to obtain a commutative algebra structure on $\IndCoh(\Ran(X/S)) \in \Cat_S$.

\begin{definition}
    Define the chiral symmetric monoidal structure to be the resulting commutative algebra object in $\Cat_S$ \[\IndCoh^{\ch}(\Ran(X/S)) = \IndCoh(\Ran(X/S),(\jmath,\sqcup))~.\] Explicitly, the tensor product is given by \[\otimes^{ch}_{i \in I} \cF_i = \sqcup_*\jmath^!\otimes_S^!\cF_i~.\]
\end{definition}

\begin{definition}
    Define the category $\Fact(X/S)$ of $S$-families of factorization algebras over $X$ to be the full subcategory \[\Fact(X/S) \subset \coComm(\IndCoh^{\ch}(\Ran(X/S)))\] spanned by coalgebras $\cA_{X/S}$ such that the maps \[\sqcup^!\cA_{X/S} \to \jmath^!\cA_{X/S}^{\otimes_S^! n}\] adjoint to the comultiplication maps \[\cA_{X/S} \to \sqcup_*\jmath^!\cA_{X/S}^{\otimes_S^! n}\] are isomorphisms for each $n > 0$.
\end{definition}

Given a smooth family $X/S$, a module over $\Ran(X/S)$ is an $S$-prestack $\cZ$ and a commutative action map \[\cZ \times_S \Ran(X/S) \xleftarrow{\jmath_{\cZ}} \Act_{\cZ} \xrightarrow{\sqcup_{\cZ}} \cZ\] for some prestack $\Act_{\cZ}$, compatible with the multiplication correspondence for $\Ran(X/S)$. We get an induced $\IndCoh(\Ran(X/S))$-module structure on $\IndCoh(\cZ)$, and therefore a notion of $\cA_{X/S}$-comodules in $\IndCoh(\cZ)$ for any factorization algebra $\cA_{X/S}$ over $X/S$.

\begin{definition}
    An object $\cM \in \cA\mh\coMod(\IndCoh(\cZ))$ is called a factorization $\cA$-module over $\cZ$ if $\sqcup_{\cZ,*}$ is right adjoint to $\sqcup_{\cZ}^!$, and the maps \[\sqcup_{\cZ}^!\cM \to \jmath_{\cZ}^!(\cM \otimes_{S}^! \cA^{\otimes_S^! n})\] adjoint to the coaction maps are isomorphisms for any $n > 0$. We denote the resulting category by $\cA\mh\FactMod(\cZ)$.
\end{definition}

\begin{example}
    Given a point $\overline{x} \in \Ran(X)(k)$, let $\cZ = \Ran_{\overline{x}}X$ the pointed Ran space, classifying finite subsets of $X$ that contains $\overline{x}$. Then $\Ran_{\overline{x}}X$ is naturally a factorization module space over $\Ran X$ with respect to the (disjoint) union map. We denote the resulting category by \[\cA\mh\FactMod_{\overline{x}} = \cA\mh\FactMod(\Ran_{\overline{x}}X)~.\] This is the category of factorization modules supported at $\overline{x}$.
\end{example}

Factorization algebras and modules admit a dual description in terms of chiral algebras and modules, a certain subcategory of Lie algebra objects in $\IndCoh^{\ch}(\Ran(X/S))$. The equivalence is given via the Koszul duality between commutative coalgebras and Lie algebras~\cite{FG}. In particular, given a factorization module supported at $\overline{x} \in \Ran X$, its fiber at $\overline{x}$ admits a structure of a chiral module supported at $\overline{x}$.

The global counterpart of a factorization algebra is given by its factorization homology:
\begin{definition} \label{def:fact-homology}
    For a smooth and proper family of curves $X/S$ and a factorization algebra $\cA_{X/S} \in \Fact(X/S)$, define the factorization homology of $\cA_{X/S}$ over $X/S$ to be the sheaf over $S$ defined by \[\int_{X/S} \cA \coloneqq p_!\cA \in \IndCoh(S)\] where $p : \Ran(X/S) \to S$ is the projection. More generally, for $\overline{x} \in \Ran(X/S)$ and $\cM \in \cA\mh\FactMod_{\overline{x}}$, define the factorization homology of $\cA_{X/S}$ over $X/S$ with coefficients in the chiral module $\cM_{\overline{x}}$ at $\overline{x}$ to be the sheaf over $S$ defined by $S$ \[\int_{(X/S,\overline{x})}(\cA;\cM_{\overline{x}}) \coloneqq p_{\overline{x},!}\cM \in \IndCoh(S)\] where $p_{\overline{x}} : \Ran_{\overline{x}}(X/S) \to S$ is the projection.
\end{definition}

\subsection{Vertex algebras} \label{subsec:VOA}

\begin{definition} \label{def:vertex-algebra}
    A vertex algebra consists of the data of:
    \begin{enumerate}
        \item A $k$-vector space $V$.
        \item A $k$-derivation $T : V \to V$.
        \item A vacuum vector $\pmb{1} \in V$.
        \item A map \[Y : V \to \operatorname{End}(V)[\![z^{\pm 1}]\!]\] sending a vector $v \in V$ to the series $v(z) \eqqcolon \sum_{n \in \ZZ} v_{(n)}z^{-n-1}$ for $v_{(n)} \in \operatorname{End}(V)$.
    \end{enumerate}
    such that the following conditions hold:
    \begin{enumerate}
        \item $\pmb{1}(z) = \id_V$ and $v(z)\pmb{1} \in v + zV[\![z]\!]$ for all $v \in V$.
        \item $u(z)v \in V((z))$ for all $u,v \in V$.
        \item For all $a,b,c \in \ZZ$ and all $u,v \in V$, \[\sum_{i \geq 0} (-1)^i\binom{c}{i}(u_{(a+c-i)}v_{(b+i)} - (-1)^{c}v_{(b+c-i)}u_{(a+i)}) = \sum_{i \geq 0}\binom{a}{i}(u_{(c+i)}v)_{(a+b-i)}\]
    \end{enumerate}
\end{definition}

Condition (3) is the vertex algebra analog of the Jacobi identity for Lie algebras.

\begin{definition}
    A vertex operator algebra (VOA) is a vertex algebra $(V,\pmb{1},T,Y)$, together with a $\ZZ$-grading on $V$ with $\dim V_n < \infty$, $V_n = 0$ for $n \ll 0$, and a conformal vector $\omega \in V$, such that the following identities hold: Let $L_n^V \coloneqq \omega_{(n+1)}$. Then
    \begin{enumerate}
        \item \[[L_m^V,L_n^V] = (m-n)L_{m+n}^V + \delta_{m+n,0}\frac{m^3-m}{12}c\] for some constant $c \in k$ called the \textit{central charge}.
        \item $L_0^V \mid V_n = n\id$.
        \item $(L_{-1}^Vu)(z) = \frac{d}{dz}u(z)$ for all $u \in V$.
    \end{enumerate} A quasi-conformal vertex algebra is a vertex algebra together with the above structure, but with a conformal vector $\omega$ replaced with operators $L_n^V$ defined only for $n\geq-1$.
\end{definition}

Condition (1) ensures that the operators $\omega_{(n+1)}$ satisfy the commutation relations corresponding to the generators $L_n = -z^{n+1}\partial_z$ of the Virasoro algebra $\operatorname{Vir}$.

\begin{example} \label{ex:KM_VOA}
    Let $\fg$ be a simple Lie algebra and $\kappa$ an invariant non-degenerate bilinear form. Then one can form the affine Kac-Moody algebra $\hat{\fg}_{\kappa}$, an infinite dimensional Lie algebra given by a central extension $\fg((t)) \oplus k \cdot \pmb{c}$ of the loop algebra $\fg((t))$, with the Lie bracket being determined by $\kappa$. To $\hat{\fg}_{\kappa}$ one can associate a module \[\VV_{\fg,\kappa} \coloneqq \operatorname{Ind}_{\fg[\![t]\!] \oplus k \cdot \pmb{c}}^{\hat{\fg}_{\kappa}} k\] Then $\VV_{\fg,\kappa}$ admits a structure of a vertex algebra, with \[(X \otimes t^{-1})(z) = \sum_{n \in \ZZ} (X \otimes t^{n})z^{-n-1}\] for an element of the form $X \otimes t^{-1} \in t^{-1}\fg \subset \VV_{\fg,\kappa}$. For a general element, the field $v(z)$ can be obtained by successive application of normally ordered product, see~\cite{FBZ}. Away from the critical level $\kappa \neq \kappa_{\operatorname{crit}} = -\frac12 \operatorname{Kil}$, where $\operatorname{Kil}$ is the Killing form, $\VV_{\fg,\kappa}$ admits a VOA structure, with the conformal vector given by the Segal-Sugawara construction. At critical level, one only has a quasi-conformal structure.
\end{example}

\begin{definition}
    A weak module for a quasi-conformal vertex algebra $V$ is a vector space $M$ together with a linear map \[Y^M : V \to \operatorname{End}M[\![z^{\pm 1}]\!]\] sending $v \in V$ to the field $v^M(z) = \sum_{n \in \ZZ} v^M_{(n)}z^{-n-1}$, and satisfying conditions analogous to those in Definition~\ref{def:vertex-algebra}. A $V$-module is a weak module $M$ together with a $\ZZ_{\geq 0}$ grading $M = \bigoplus_{n \geq 0}M_n$ satisfying \[v_{(n)}^M M_m \subset M_{m+\deg(v)-n-1}\] for all homogeneous $v \in V_{\deg v}$.
\end{definition}

\begin{example} \label{ex:KM-modules}
    For $V = \VV_{\fg,\kappa}$ as in Example~\ref{ex:KM_VOA}, the data of a $V$-module is equivalent to that of a smooth $\hat{\fg}_{\kappa}$-module, or equivalently a discrete $\hat{U}_{\kappa}(\hat{\fg})$-module, where $\hat{U}_{\kappa}(\hat{\fg})$ is the quotient of the completed universal enveloping algebra $\hat{U}(\hat{\fg}_{\kappa})$ by the ideal $(1 - \pmb{c})$.
\end{example}

More generally, for any vertex algebra $V$, the category of $V$-modules is equivalent to that of $U(V)$-modules, where $U(V)$ is a topological associative algebra defined as follows:

\begin{definition} \label{def:UV}
    For a quasi-conformal vertex algebra $V$, let $\fL(V)$ be the Lie algebra whose underlying vector space is given by $\Gamma_{\dR}(\GG_m, V \otimes \omega_{\GG_m})$ with respect to the connection $\partial_z = T \otimes \id + \id \otimes \frac{d}{dz}$, so that it is generated by $v_{[n]} \coloneqq v \otimes z^ndz$, and subject to the relations \[\partial_z v_{[n]} = (Tv)_{[n]} + nv_{[n-1]} = 0~.\] The Lie bracket is defined by \[[u_{[m]},v_{[n]}] = \sum_{i \geq 0} \binom{m}{i}(u_{(i)}v)_{[m+n-i]}~.\] It admits a grading given by \[\deg(v_{[n]}) = \deg(v)-n-1\] Let $U(\fL(V))$ be the completion of its universal enveloping algebra with respect to the filtration \[U(\fL(V))^{\leq k} = \bigoplus_{n \in \ZZ} \bigoplus_{i \leq k} U(\fL(V))_{n-i} \cdot U(\fL(V))_{i} = U(\fL(V)) \cdot U(\fL(V))_{\leq k}\] where $U(\fL(V)) = \bigoplus_{n \in \ZZ} U(\fL(V))_n$ is the grading induced by the grading on $\fL(V)$. Finally, let $U(V)$ be the quotient of $U(\fL(V))$ by the relations \[\pmb{1}_{[i]} = \delta_{i,-1}\] and the Jacobi identity \begin{equation} \label{eq:Jac_id}
        \sum_{i \geq 0} (-1)^i\binom{c}{i}(u_{[a+c-i]}v_{[b+i]} - (-1)^{c}v_{[b+c-i]}u_{[a+i]}) = \sum_{i \geq 0}\binom{a}{i}(u_{(c+i)}v)_{[a+b-i]}
    \end{equation} for all $a,b,c \in \ZZ$ and $u,v \in V$.
\end{definition}

The grading on $U(\fL(V))$ induces a grading on $U(V)$ \begin{equation} \label{eq:UV-grading}
    U(V) \simeq \bigoplus_{n \in \ZZ} U(V)_n~.
\end{equation} In particular, $U(V)_0 \subset U(V)$ is a subalgebra. It admits a filtration \[\dotsb \subset U(V)_0^{\leq k} \subset U(V)_0^{\leq k+1} \subset \dotsb\] defined by \begin{equation} \label{eq:UV-filt}
    U(V)_0^{\leq k} = \bigoplus_{i \leq k} U(V)_{-i}\hat{U}(V)_i = (U(V) \cdot U(V)_{\leq k})_0
\end{equation}

\begin{definition} \label{def:Zhu}
    Define the Zhu algebra associated to $V$ to be the quotient \[A(V) = U(V)_0 / U(V)_0^{< 0}\]
\end{definition}

\begin{remark}
    The above definition is the one proven in~\cite{FZ}. The original definition in~\cite{Zhu} is given by a certain quotient of $V$.
\end{remark}

Using the algebra $A(V)$, Damiolini, Gibney, and Krashen~\cite{DGK} defined a larger associative algebra, which plays the role of the ``value attached to a node'' in their extension to the boundary of the vertex algebra bundle. We recall here the definition of its underlying $U(V)$-bimodule:
\begin{definition} \label{def:MTA}
    The mode transition bimodule $\fA(V)$ associated to a vertex algebra $V$ is given by the tensor product \[\fA(V) = (U(V)/U(V)U(V)_{< 0}) \otimes_{U(V)_0} A(V) \otimes_{U(V)_0} (U(V)/U(V)_{>0}U(V))~.\]
\end{definition}

This bimodule structure then extends to an associative algebra structure on $\fA(V)$ extending that of $A(V)$, which we will not use in this paper.

\subsection{Universal factorization algebras}

Vertex algebras describe the local structure of factorization algebras. Furthermore, given an $\operatorname{Aut}\hat{\cO}$-equivariant vertex algebra $V$, Frenkel and Ben-Zvi~\cite{FBZ} described a construction of a factorization algebra over any smooth curve, having $V$ as its local structure over each point. Factorization algebras constructed this way are known as \textit{universal factorization algebras} --- they are defined over any smooth curve, and compatible with \'{e}tale pullbacks. In this section we recall the definition given in~\cite{Lur, Naf2} of universal factorization algebras as a derived generalization of the notion of equivariant vertex algebras.

\begin{definition} \label{def:MDisk}
    For a finite set $I$, let $\MDisk_I$ be the stack of $I$-pointed multidisks, whose $R$-points are given by formal $R$-schemes $\cX$, \'{e}tale locally isomorphic to the formal completion of a finite and flat $R$-divisor in $\AA^1_R$, together with a surjective $R$-morphism \[(x_i)_{i \in I} : (\Spec R^{\red})^{\sqcup I} \to \cX~.\] For a surjective map $\alpha : I \onto J$, we have a diagonal map \[\Delta_{(\alpha)} : \MDisk_J \to \MDisk_I\] defined by \[(\cX, (x_j)_{j \in J}) \mapsto (\cX, (x_{\alpha(i)})_{i \in I})~.\] By taking colimit over all finite sets and surjection, we get the Ran space of multidisks \[\MDisk_{\Ran} \coloneqq \colim_{\fSet^{\surj,\op}} \MDisk_I~.\]
\end{definition}

Taking $I = \{\pt\}$ produces the classifying stack $B\operatorname{Aut}\hat{\cO}$, classifying formal neighborhoods of smooth points inside curves. In general, however, the stack $\MDisk_I$ is bigger than the product $B\operatorname{Aut}\hat{\cO}^I$, and contains formal completion around $I$ $R$-divisors that intersect over subsets of $\Spec R$.

In~\cite{Naf2}, we defined a sheaf of categories \[\IndCoh^!_{\ren}(\MDisk_{\Ran}) = \lim_{\fSet^{\surj}} \IndCoh^!_{\ren}(\MDisk_I)\] of renormalized ind-coherent sheaves, together with $!$-pullback along arbitrary maps.

We have a disjoint union operation \[\sqcup : \MDisk_I \times \MDisk_J \to \MDisk_{I \sqcup J}~,\] which induces a commutative monoid structure \[\sqcup : \MDisk_{\Ran} \times \MDisk_{\Ran} \to \MDisk_{\Ran}~.\]

\begin{definition}
    A universal factorization algebra is an object \[\cA \in \IndCoh^!_{\ren}(\MDisk_{\Ran})\] together with commutative isomorphisms \[\sqcup^!\cA \iso \cA^{\sqcup I}\] for each finite set $I$. We denote the resulting category by $\Fact^{\univ}$.
\end{definition}

Given a family of smooth curves $X$ over a finite type algebraic stack $S$, we have a map \[\pi_{X/S} : \Ran(X/S) \to \MDisk_{\Ran}\] sending a finite set to its formal neighborhood. More precisely, given $f : \Spec R \to S$ and $x_i : \Spec R^{\red} \to X; i \in I$, we define \[\pi_{X/S}(f,\{x_i\}) = (X_R \times_{X_{R\mh\dR}} (\Spec R^{\red})^{\sqcup I}, \{x_i\})~,\] where \[X_{R\mh\dR} = X_{\dR} \times_{\Spec R_{\dR}} \Spec R\] is the relative de Rham stack. In~\cite[Theorem~2.42]{Naf2} we prove the following:

\begin{theorem} \label{thm:loc-UFA}
    The map $\pi_{X/S}$ induces a map \[\pi_{X/S}^{!,\fact} : \Fact^{\univ} \to \Fact(X/S)~.\]
\end{theorem}

We denote by \[\cA_{X/S} \coloneqq \pi_{X/S}^{!,\fact}\cA\] the factorization algebra over $X/S$ corresponding to $\cA$.

\section{Factorization modules over nodes and punctures} \label{sec:ZA}

In this section we recall the construction of factorization modules associated to a universal factorization algebra $\cA$, supported either at a node or at a puncture, modelled as a weight-one marked point in $\bcM_{g,n}$. They are all built from modules over the chiral Zhu algebra $\fZ_{\cA}^0$, a (non-unital) associative algebra associated to $\cA$.

\subsection{Stacks of semistable modifications} \label{sec:ss-modif}

Given a curve $X$ and a point $x \in X$, a semistable modification at $x$ is a curve $X'$ and a map $X' \to X$ which is an isomorphism away from $x$, and its fiber over $x$ is given by a rational chain $\PP^1 \lor \dotsb \lor \PP^1$. We will have two classes of such modifications: The first is when $x$ is a smooth marked point, which we view as a puncture, and the second is when $x$ is a node. To define these modifications in families, we use the moduli $\fM_{g,n}^{\sst}$ of semistable $n$-pointed genus $g$ curves.

\begin{definition} \label{def:ss-modif}
    For $g,n \geq 0$ such that $2g+n-2 > 0$, let \[\fM_{g,n}^{\sst} \to \bcM_{g,n}\] be the stable reduction map. For $(X,\overline{x}) \in \bcM_{g,n}(R)$, let the stack of semistable modifications of $(X,\overline{x})$ be the fiber product \[\fM_{(X,\overline{x})}^{\sst} \coloneqq \fM_{g,n}^{\sst} \times_{\bcM_{g,n}} \Spec R\] along the map $\Spec R \to \bcM_{g,n}$ classifying $(X,\overline{x})$.
\end{definition}

As was observed in~\cite{ACFW,graberRelativeVirtualLocalization2005}, the local structure of these spaces is independent of the pointed curve, and can be modelled by the stack $\fM_{0,2}^{\sst}$ and its variants, which we define below.

\begin{definition} \label{def:punct-modif-stack}
    Let $\fM_{0,2+1}^{\sst}, \fM_{0,1+2}^{\sst} \subset \fM_{0,3}^{\sst}$ be the open substacks where the first (resp. last) two marked points are supported on the same irreducible component. The right (resp. left) puncture-modification stack $\fM_{0,2+1}^{\sst,\land}$ (resp. $\fM_{0,1+2}^{\sst,\land}$) is the formal completion of $\fM_{0,2+1}^{\sst}$ (resp. $\fM_{0,1+2}^{\sst}$) along the closed embedding $(\PP^1,0,1,\infty) \lor (-) : \fM_{0,2}^{\sst} \to \fM_{0,2+1}^{\sst}$ (resp. $(-) \lor (\PP^1,0,1,\infty) : \fM_{0,2}^{\sst} \to \fM_{0,1+2}^{\sst}$). 
\end{definition}

\begin{definition} \label{def:node-modif-stack}
    Let $\fM_{0,2+2}^{\sst} \subset \fM_{0,4}^{\sst}$ be the open substack where the first and last two marked points are supported on the same irreducible component. The node-modification stack $\fM_{0,2+2}^{\sst,\land}$ is the formal completion of $\fM_{0,2+2}^{\sst}$ along the closed embeddings \[(\PP^1,0,1,\infty) \lor (-) \lor (\PP^1,0,1,\infty) : \fM_{0,2}^{\sst} \to \fM_{0,2+2}^{\sst}~.\] The strictly nodal closed substack $\fM_{0,2+2}^{\sst,\land,\nd} \subset \fM_{0,2+2}^{\sst,\land}$ is the fiber product \[\fM_{0,2+2}^{\sst,\land,\nd} \coloneqq \fM_{0,2+2}^{\sst,\land} \times_{\AA^1} \{0\}~,\] where the map \[\fM_{0,2+2}^{\sst} \to \bcM_{0,2+2} \simeq \AA^1 \subset \PP^1 \simeq \bcM_{0,4}\] is the stable reduction map, and $\{0\}$ is the nodal locus.
\end{definition}

Over each $\fM_{g,n}^{\sst}$, one has a universal curve $\fX_{g,n}^{\sst}$. In particular, one gets the relative Ran spaces $\Ran(\fX_{g,n}^{\sst}/\fM_{g,n}^{\sst})$. We want to restrict to a subspace which contains only smooth subsets. Furthermore, in order to get a pseudo-proper prestack, we need to restrict to stable configurations, defined as follows:

\begin{definition} \label{def:MgnRan}
    For a finite set $I$, let $\bcM_{g,n,I}$ be the moduli of genus $g$ curves $X$ together with $n + |I|$ marked points $x_1,\dots,x_n,y_1,\dots,y_{|I|}$, such that the weighted pointed curve \[(X,1 \cdot (x_1 + \dotsb + x_n) + \frac{1}{|I|+1} \cdot (y_1 + \dotsb + y_{|I|}))\] is stable in the sense of~\cite{Has}. Define \[\bcM_{g,n,\Ran} \coloneqq \colim_{I \in \fSet^{\surj,\op}} (\bcM_{g,n,I})_{\bcM_{g,n}\mh\dR}\] where the colimit is taken with respect to diagonal embeddings.
\end{definition}

The spaces $\bcM_{g,n,I}$ are smooth and proper schemes, and so their colimit is pseudo-proper. By restriction, we get a space $\bcM_{0,2+1,\Ran}$.

\begin{definition} \label{def:M02plus1Ran}
    The Ran space of right puncture-modification is the sub-prestack \[\bcM_{0,2+1,\Ran}^{\land} \subset \bcM_{0,2+1,\Ran}\] given by the formal completion along the closed embedding \[(\PP^1,0,1,\infty) \lor (-) : \bcM_{0,2,\Ran} \to \bcM_{0,2+1,\Ran}~.\] The Ran space of left puncture-modification $\bcM_{0,1+2,\Ran}^{\land}$, the node-modification Ran space $\bcM_{0,2+2}^{\land}$, and its strictly nodal subspace $\bcM_{0,2+2,\Ran}^{\land,\nd}$ are defined similarly, as in Definitions~\ref{def:punct-modif-stack} and~\ref{def:node-modif-stack}.
\end{definition} 

The projection \[\bcM_{0,2,\Ran} \to \fM_{0,2}^{\sst}\] can be upgraded to a map of non-unital associative algebras, with the binary operation in both cases given by gluing along marked points \[(X,x_0,x_{\infty}) \lor (Y,y_0,y_{\infty}) = (X \underset{x_{\infty} \sim y_0}{\cup} Y, x_0, y_{\infty})~.\]

Any semistable modification at a point is given by inserting a rational chain, and so the spaces $\fM_{g,n}^{\sst}, \bcM_{g,n,\Ran}$ are modelled locally by the spaces of semistable modifications for genus zero curves. More precisely, we have the following results (see~\cite[Theorem~1.3.2]{ACFW} and~\cite[Proposition~3.23, 3.26]{Naf2}):
\begin{proposition} \label{prop:local-str-ss-modif}
    \begin{enumerate}
        \item For a smooth pointed curve $(X,x) \in \cM_{g,1}(k)$ there is an isomorphism \[\fM_{(X,x)}^{\sst} \simeq \fM_{0,2+1}^{\sst} \simeq \fM_{0,1+2}^{\sst}\] and a Cartesian square \[\begin{tikzcd} {\bcM_{0,2+1,\Ran}^{\land}} & {\bcM_{(X,x),\Ran}} \\
        {\{x\}} & {\Ran X}
        \arrow[from=1-1, to=1-2]
        \arrow[from=1-1, to=2-1]
        \arrow[from=1-2, to=2-2]
        \arrow[from=2-1, to=2-2]
        \end{tikzcd}\]
        \item For a family of smooth curves $X/B$ over a smooth curve $B$ with a single node $q \in X_b$ over a closed point $b \in B$ and a smooth general fiber, there is an isomorphism \[\fM_{X/B}^{\sst} \simeq \fM_{0,2+2}^{\sst}/\GG_m \times_{\AA^1/\GG_m} B\] and a Cartesian square \[\begin{tikzcd}
        {\bcM_{0,2+2,\Ran}^{\land}} & {\bcM_{(X,x),\Ran}} \\
        {(X_{B\mh\dR})^{\land}_{p}} & {\Ran (X/B)}
        \arrow[from=1-1, to=1-2]
        \arrow[from=1-1, to=2-1]
        \arrow[from=1-2, to=2-2]
        \arrow[from=2-1, to=2-2]
        \end{tikzcd}\] which restricts to a Cartesian square \[\begin{tikzcd}
	    {\bcM_{0,2+2,\Ran}^{\land,\nd}} & {\bcM_{(X,x),\Ran}} \\
	    {\{p\}/\{b\}} & {\Ran (X/B)}
	    \arrow[from=1-1, to=1-2]
	    \arrow[from=1-1, to=2-1]
	    \arrow[from=1-2, to=2-2]
	    \arrow[from=2-1, to=2-2]
        \end{tikzcd}\]
    \end{enumerate}
\end{proposition}

Following~\cite{Li,ACFW,NY}, we can give an explicit description of the stacks $\fM_{0,2+2}^{\sst}$, $\fM_{0,2+1}^{\sst}$ etc., as well as their universal curves. We will use the arguments given in~\cite[Section~3.2]{Naf2}.

\begin{notation}
    For $1 \leq i,j \leq \ell$, denote $t_{[i,j]} = t_it_{1+1}\dotsb t_j$, with $t_{[i,j]}=1$ if $i > j$.
\end{notation}

\begin{definition} \label{def:A[l]}
    For $\ell \geq 0$, define the quotient stacks \begin{align*}
        \fA[\ell] & \coloneqq \{0\} \overset{\GG_m}{\times} \AA^1 \overset{\GG_m}{\times} \dotsb \overset{\GG_m}{\times} \AA^1 \overset{\GG_m}{\times} \{0\} = \AA^{\ell}/ \GG_m^{\ell+1} \\
        \fA[\ell + \epsilon] & \coloneqq \{0\} \overset{\GG_m}{\times} \AA^1 \overset{\GG_m}{\times} \dotsb \overset{\GG_m}{\times} \AA^1 \overset{\GG_m}{\times} \hat{\AA}^1 = (\AA^{\ell} \times \hat{\AA}^1)/ \GG_m^{\ell+1} \\
        \fA[\epsilon + \ell] & \coloneqq \hat{\AA}^1 \overset{\GG_m}{\times} \AA^1 \overset{\GG_m}{\times} \dotsb \overset{\GG_m}{\times} \AA^1 \overset{\GG_m}{\times} \{0\} = (\hat{\AA}^1 \times \AA^{\ell})/ \GG_m^{\ell+1} \\
        \fA[\epsilon + \ell + \epsilon] & \coloneqq \hat{\AA}^1 \overset{\GG_m}{\times} \AA^1 \overset{\GG_m}{\times} \dotsb \overset{\GG_m}{\times} \AA^1 \overset{\GG_m}{\times} \hat{\AA}^1 = (\hat{\AA}^1 \times \AA^{\ell} \times \hat{\AA}^1) / \GG_m^{\ell+1}~.
    \end{align*}
\end{definition}

\begin{proposition} \label{prop:M0-ss-str}
    For $\ell\geq 0$, there are \'{e}tale covers \[\fA[\epsilon + \ell + \epsilon] \to \fM_{0,2+2}^{\sst,\land,(\leq \ell+3)} \into \fM_{0,2+2}^{\sst,\land}\] of the open substack of at most $\ell+3$ irreducible components. For a surjective order-preserving map $\alpha : [\ell] \onto [\ell']$, we have maps \[\fA[\epsilon + \ell' + \epsilon] \to \fA[\epsilon + \ell + \epsilon]\] given by \[(t_0,t_1,\dots,t_{\ell'},t_{\ell'+1}) \mapsto (t_0,t_{[\alpha(0)+1,\alpha(1)]},\dots,t_{[\alpha(\ell-1)+1,\alpha(\ell)]},t_{\ell+1})\] and an induced isomorphism \[\fM_{0,2+2}^{\sst,\land} \simeq \colim_{\Delta^{\surj,\op}} \fA[\epsilon + \ell + \epsilon]\] 
    
    The images of the closed embeddings \[(-) \lor (\PP^1,0,1,\infty) : \fM_{0,2+1}^{\sst,\land} \into \fM_{0,2+2}^{\sst,\land}\] and \[(\PP^1,0,1,\infty) \lor (-) : \fM_{0,1+2}^{\sst,\land} \into \fM_{0,2+2}^{\sst,\land}\] can be identified with the colimit over $\fA[\epsilon + \ell]$ and $\fA[\ell + \epsilon]$ respectively, and the image of the closed embedding \[(\PP^1,0,1,\infty) \lor (-) \lor (\PP^1,0,1,\infty) : \fM_{0,2}^{\sst} \into \fM_{0,2+2}^{\sst,\land}\] can be identified with the colimit over $\fA[\ell]$.
\end{proposition}

The stack $\fA[\ell + 2]$ can be viewed as the classifying stack for curves with a choice of $\ell+3$ ordered irreducible components, where the $i$-th component coincides with the $i+1$-th component over the locus $\{(t_0,\dots,t_{\ell+2}); t_i \neq 0\}$, and the substack $\fA[\epsilon + \ell + \epsilon]$ as the substack where the first and last component are, up to a deformation, separated by a node from the second and second-to-last components respectively. In particular, over $(\hat{\AA}^1 \times \GG_m^{\ell} \times \hat{\AA}^1) / \GG_m^{\ell+1}$ we have deformations of the curve $\PP^1 \lor \PP^1 \lor \PP^1$, with $0,1$ of the first copy and $1,\infty$ of the last copy being the four marked points, while over $\{(0,\dots,0)\}$ we have the curve given by the join of $\ell+3$ copies of $\PP^1$.

The universal curve over $\fA[\ell]$, as well as its formal completions over $\fA[\epsilon + \ell]$, $\fA[\ell + \epsilon]$, and $\fA[\epsilon + \ell + \epsilon]$, can be described as follows.

In~\cite{Li}, Li defined a family of two-pointed genus zero curves $\overline{W}[\ell] \to \AA^{\ell}$ for each $\ell > 0$, whose fiber over a point $(t_1,\dots,t_{\ell})$ with $i$ zeros has exactly $i$ nodes. Following the description in~\cite{NY}, we have an open cover \[\overline{W}[\ell] \simeq \bigcup_{1 \leq i \leq \ell} \overline{V}^{i}[\ell]\] where the sets $\overline{V}^i[\ell]$ are constructed as follows: First, define \[V^i[\ell] = \left\{(x_i,y_i,t_1,\dots,t_{\ell}) \vcentcolon x_iy_i = t_i\right\} \subset \AA^{\ell+2} \xrightarrow{(t_1,\dots,t_{\ell})} \AA^{\ell}\] 
so that $x_i$ and $y_i$ give coordinates for the $\{i-1,i\}$-th irreducible components, where we use the convention that the $i-1$-th and $i$-th components are the same if $t_i \neq 0$. Then, define $\overline{V}^i[\ell] = V^i[\ell]$ for $1 < i < \ell$, $\overline{V}^1[\ell]$ the closure of $V^1[\ell]$ in $\PP^1 \times \AA^{\ell+1}$ (namely, adding the point $0$), and $\overline{V}^{\ell}[\ell]$ the closure of $V^{\ell}[\ell]$ in $\AA^1 \times \PP^1 \times \AA^{\ell}$ (namely, adding the point $\infty$).

The intersection of these open subsets is given as follows: For $i < j$ let \[V^{i,j}[\ell] \coloneqq \{y_it_{[i+1,j-1]} \neq 0\} \subset V^i[\ell]\] and \[V^{j,i}[\ell] \coloneqq \{x_jt_{[i+1,j-1]} \neq 0\} \subset V^j[\ell]~.\] Then we can identify \[V^i[\ell] \cap V^j[\ell] \simeq V^{i,j}[\ell] \simeq V^{j,i}[\ell]\] where the second isomorphism is given by the change of coordinates \[(x_i,y_i,t_1,\dots,t_{\ell}) \mapsto (t^{-1}_{[i+1,j-1]}y^{-1}_i, t_{[i+1,j]}y_i, t_1,\dots,t_{\ell})\]

We define a $\GG^{\ell+1}_m$ action on $\overline{W}[\ell]$ by defining it on each $\overline{V}^i[\ell]$, and showing that the gluing maps are compatible with the action. Let \[(\lambda_0,\dots,\lambda_{\ell}) \cdot (x_i,y_i,t_1,\dots,t_{\ell}) = (\lambda_{i-1}^{-1}x_i,\lambda_iy_i,\lambda_0^{-1}t_1\lambda_1,\dots,\lambda_{\ell-1}^{-1}t_{\ell}\lambda_{\ell})\]
Then over $V^{i,j}[\ell] \iso V^{j,i}[\ell]$ we have \begin{align*}
    (\lambda_0,\dots,\lambda_{\ell}) \cdot (x_i,y_i,t_1,\dots,t_{\ell}) & \mapsto (\lambda_i^{-1}t^{-1}_{[i+1,j-1]}y^{-1}_i, \lambda_it_{[i+1,j]}y_i, \lambda_0^{-1}t_1\lambda_1,\dots,\lambda_{\ell-1}^{-1}t_{\ell}\lambda_{\ell}) \\
    & = (\lambda_0,\dots,\lambda_{\ell}) \cdot (t^{-1}_{[i+1,j-1]}y^{-1}_i, t_{[i+1,j]}y_i, t_1,\dots,t_{\ell})
\end{align*}

\begin{definition}
    For $\ell \geq 0$, let \begin{align*}
        \overline{W}[\epsilon + \ell + \epsilon] & = (\overline{W}[\ell+2])^{\land}_{\overline{W}[\ell+2] \underset{\AA^{\ell+2}}{\times} (\{0\} \times \AA^{\ell} \times \{0\})} \\
        \overline{W}[\ell + \epsilon] & = (\overline{W}[\ell+1])^{\land}_{\overline{W}[\ell+1] \underset{\AA^{\ell+1}}{\times} (\AA^{\ell} \times \{0\})} \\
        \overline{W}[\epsilon + \ell] & = (\overline{W}[\ell+1])^{\land}_{\overline{W}[\ell+1] \underset{\AA^{\ell+1}}{\times} (\{0\} \times \AA^{\ell})} \\ 
    \end{align*}
    For $\ell=0$, we will write \[\overline{W}[0] = \PP^1 \to \AA^0 = \pt~.\]
    We then get families of curves over the stacks in Definition~\ref{def:A[l]}: \begin{align*}
        \fX[\epsilon + \ell + \epsilon] & = \overline{W}[\epsilon + \ell + \epsilon] / \GG_m^{\ell+1} \to \fA[\epsilon + \ell + \epsilon]\\
        \fX[\ell + \epsilon] & = \overline{W}[\ell + \epsilon] / \GG_m^{\ell+1}  \to \fA[\ell + \epsilon]\\
        \fX[\epsilon + \ell] & = \overline{W}[\epsilon + \ell] / \GG_m^{\ell+1}  \to \fA[\epsilon + \ell]\\
        \fX[\ell] & = \overline{W}[\ell] / \GG_m^{\ell+1}  \to \fA[\ell]\\
    \end{align*}
\end{definition}

\begin{definition} \label{def:univ-punct-sm-curve}
    Let \[\mathring{\fX}_{0,2}^{\sst} \coloneqq \fX_{0,2}^{\sst,\sm} \backslash (\fM_{0,2}^{\sst} \sqcup \fM_{0,2}^{\sst}) \to \fM_{0,2}^{\sst}\] be the smooth, punctured universal curve, where the map $\fM_{0,2}^{\sst} \sqcup \fM_{0,2}^{\sst} \to \fX_{0,2}^{\sst}$ is given by the sections $0,\infty$. Let \[W[\ell]/\GG_m^{\ell+1} = \mathring{\fX}[\ell] = \fX[\ell] \times_{\fX_{0,2}^{\sst}} \mathring{\fX}_{0,2}^{\sst} \to \AA^{\ell}\] be the corresponding base-change. Similarly, define $\mathring{\fX}_{0,2+1}^{\sst,\land}$ (resp. $\mathring{\fX}_{0,1+2}^{\sst,\land}$) as the smooth locus of the universal curve with the section $\infty$ (resp. $0$) removed, and $\mathring{\fX}_{0,2+2}^{\sst,\land}$ the smooth locus of the universal curve. Let $\mathring{\fX}[\ell + \epsilon]$, $\mathring{\fX}[\epsilon + \ell]$, and $\mathring{\fX}[\epsilon + \ell + \epsilon]$ be the corresponding base-change.
\end{definition}

\begin{definition}
    For $0 < i < j \leq \ell$, let $U^{i,j}[\ell] = V^{i,j}[\ell] = V^{j,i}[\ell]$. Define \[U^{0,j}[\ell] = \{x_jt_{[1,j-1]} \neq 0\} \subset V^j[\ell]\] and \[U^{i,\ell+1}[\ell] = \{y_it_{[i+1,\ell]} \neq 0\} \subset V^i[\ell]\] Set $y_0 = x_1^{-1}$, and $x_{\ell+1} = y_{\ell}^{-1}$.
\end{definition}

We then have:
\begin{lemma} \label{lem:cover-sm-punct}
     The smooth and punctured locus admits the following covering:
    \[\mathring{\fX}[\ell] \simeq (\bigcup_{i=0}^{\ell} U^{i,i+1}[\ell]) / \GG_m^{\ell+1}\]
\end{lemma} 

Therefore, $\mathring{\fX}[\ell]$ has a cover by a $\GG_m^{\ell+1}$ quotients of $\ell+1$ affine schemes isomorphic to $\GG_m \times \AA^1$, with coordinates $(x_{i}, t_{i}) = (y_{i-1}^{-1},t_i)$ on the $i+1$-th component.

\subsection{Universal factorization algebras over semistable modifications}

Given a universal factorization algebra $\cA$, we get from the construction of Theorem~\ref{thm:loc-UFA} a factorization algebra over the smooth and punctured universal curve $\mathring{\fX}_{g,n}^{\sst} / \fM_{g,n}^{\sst}$. By restriction, we get a sheaf $\Vac(\cA)_{g,n}^{\epsilon}$ over $\bcM_{g,n,\Ran}$. The stable reduction map \[\fM_{g,n}^{\sst} \to \bcM_{g,n}\] lifts to a map \[\bcM_{g,n,\Ran} \to \Ran_{g,n} \coloneqq \Ran(\bcX_{g,n}/\bcM_{g,n})~.\] Pushing forward $\Vac(\cA)_{g,n}^{\epsilon}$, we get a sheaf $\Vac(\cA)_{g,n}^{0}$. Define an open subspace \[\mathring{\Ran}_{g,n} \subset \Ran_{g,n}\] whose fiber over an $R$-point $(X,\overline{x})$ is given by $\Ran(X^{\sm} \backslash \{\overline{x}\}/R)$. The latter is the relative Ran space of a smooth family of curves, so we have a factorization algebra $\cA_{g,n}$ over it. As we prove in~\cite{Naf2}, $\Vac(\cA)_{g,n}$ admits a structure of a factorization module over $\cA_{g,n}$, whose fiber over smooth unmarked points is given by the vacuum module. This defines our extension of factorization algebras to nodal and punctured curves. Its fiber over a puncture is then a chiral module supported at a point, and its fiber over a node admits a structure of a chiral bimodule supported at a point, with the two commuting actions coming from the two components of the normalization. They can be described as follows.

\begin{definition} \label{def:smoothing-module}
    For a universal factorization algebra $\cA$, let $\Vac(\cA)_{0,2+2}^{\epsilon}$ be the restriction of $\Vac(\cA)_{0,4}^{\epsilon}$ to $\bcM_{0,2+2,\Ran}$. Define $\Vac(\cA)_{0,2+1}$ and $\Vac(\cA)_{0,1+2}$ similarly. The smoothing module $\tilde{\fZ}_{\cA}$ of $\cA$ is the space of compactly supported sections \[\tilde{\fZ}_{\cA} \coloneqq \Gamma_c(\bcM_{0,2+2,\Ran}^{\land}, \Vac(\cA)^{\epsilon}_{0,2+2})~.\] The node module $\fZ_{\cA}$ of $\cA$ is the subspace \[\fZ_{\cA} \coloneqq \Gamma_c(\bcM_{0,2+2,\Ran}^{\nd}, \Vac(\cA)^{\epsilon}_{0,2+2})~.\] The right and left puncture modules $\fZ_{\cA}^+$ and $\fZ_{\cA}^-$ are the spaces \begin{align*}
        \fZ_{\cA}^+ & \coloneqq \Gamma_c(\bcM_{0,2+1,\Ran}^{\land}, \Vac(\cA)^{\epsilon}_{0,2+1}) \\
        \fZ_{\cA}^- & \coloneqq \Gamma_c(\bcM_{0,1+2,\Ran}^{\land}, \Vac(\cA)^{\epsilon}_{0,1+2})~.
    \end{align*} Finally, the chiral Zhu algebra $\fZ_{\cA}^0$ of $\cA$ is the space \[\fZ_{\cA}^0 \coloneqq \Gamma_c(\bcM_{0,2,\Ran}, \Vac(\cA)^{\epsilon}_{0,2})~.\] The prestacks $\bcM_{0,2+2,\Ran}^{\land}$ etc. are as in Definition~\ref{def:M02plus1Ran}.
\end{definition}

By Proposition~\ref{prop:local-str-ss-modif}, up to a choice of a local coordinate, the fiber of $\Vac(\cA)_{g,n}^0$ at a marked point is given by $\fZ_{\cA}^+ \simeq \fZ_{\cA}^-$, the fiber over a node given by $\fZ_{\cA}$, and the fiber over a formal neighborhood of a node in a smoothing family (namely, a family of curves which is given locally around the node by $\Spf k[\![x,y,t]\!]/(xy-t)$) is given by $\tilde{\fZ}_{\cA}$.

The associative algebra structure on $\bcM_{0,2,\Ran}$ induces an associative algebra structure on $\fZ_{\cA}^0$. In Section~\ref{sec:ss} we will prove that its zeroth homology recovers the usual Zhu algebra in the case where $\cA$ comes from a quasi-conformal vertex algebra.

We then define the factorization homology by pushforward:

\begin{definition} \label{def:fact-hom}
    For $2g+n-2>0$ and a universal factorization algebra $\cA$, let \[\int_{g,n} \cA \coloneqq p^{g,n}_!\Vac(\cA)^0_{g,n} \simeq q^{g,n}_!p^{g,n}_!\Vac(\cA)^{\epsilon}_{g,n} \in \IndCoh(\bcM_{g,n})~,\] where \[\bcM_{g,n,\Ran} \xrightarrow{q_{g,n}} \Ran_{g,n} \xrightarrow{p_{g,n}} \pt\] are the projections. For $(X,\overline{x}) \in \bcM_{g,n}(R)$, define \[\int_{X \backslash {\overline{x}}} \cA = \int_{g,n} \cA \mid (X,\overline{x}) \in \IndCoh(R)\] to be the $!$-restriction.
\end{definition}

\subsection{Unit factorization algebra}

Our goal in this section is to compute the spaces $\fZ,\fZ^0$ etc. for the case of the unit factorization algebra.

\begin{definition}
    Let $\pmb{1}$ be the universal factorization algebra, given by $\omega_{\MDisk_{\Ran}}$. For a family of smooth curves $X/S$, let $\pmb{1}_{X/S}$ be the corresponding factorization algebra, whose underlying D-module is given by $\omega_{\Ran(X/S)}$. 
\end{definition}

For a smooth and proper curve $X$, a basic result in the theory of chiral algebras over $X$ is the contractibility of the Ran space $\Ran X$. More precisely, as was proven in~\cite[Proposition~4.3.3]{BD}, the trace map \[\int_X \pmb{1}_X = \Gamma_{\dR,c}(\Ran X, \omega_{\Ran X}) \to k\] is an isomorphism. 

\begin{proposition}\label{prop:Z-unit}
    There is an isomorphism $\tilde{\fZ}_{\pmb{1}} \simeq \omega_{k[![t]\!]}$, which restrict to isomorphisms $\fZ_{\pmb{1}}^0 \simeq \fZ_{\pmb{1}} \simeq \fZ_{\pmb{1}}^{\pm 1} \simeq k$.
\end{proposition}

The proof structure is the following: The spaces $\tilde{\fZ}_{\pmb{1}}$, $\fZ_{\pmb{1}}^{+}$, and $\fZ_{\pmb{1}}^0$ are defined as (compactly supported) global sections over $\bcM_{0,2+2,\Ran}^{\land}$, $\bcM_{0,1+2,\Ran}^{\land}$, and $\bcM_{0,2,\Ran}$ respectively. The projections to a point factor as \begin{align*}
    \bcM_{0,2+2,\Ran}^{\land} & \xrightarrow{q^{2+2}} \fM_{0,2+2}^{\sst,\land} \xrightarrow{p^{2+2}} \pt \\
    \bcM_{0,1+2,\Ran}^{\land} & \xrightarrow{q^{1+2}} \fM_{0,1+2}^{\sst,\land} \xrightarrow{p^{1+2}} \pt \\
    \bcM_{0,2,\Ran} & \xrightarrow{q^{2}} \fM_{0,2}^{\sst} \xrightarrow{p^{2}} \pt~,
\end{align*} and so $\tilde{\fZ}_{\pmb{1}} \simeq q^{2+2}_!p^{2+2}_!\omega_{\bcM_{0,2+2,\Ran}^{\land}}$, $\fZ_{\pmb{1}} \simeq q^{1+2}_!p^{1+2}_!\omega_{\bcM_{0,1+2,\Ran}^{\land}}$, and $\fZ_{\pmb{1}}^0 \simeq p^2_!q^2_!\omega_{\bcM_{0,2,\Ran}}$. Since fibers of $q^{2+2}$ are de Rham spaces, $q^{2+2}_!$ is well-defined on relative holonomic D-modules, and in particular on $\omega$, and similarly for $q^{1+2}_!$ and $q^2_!$. As for $p^{2+2}_!$, $p_!^{1+2}$, and $p_!^2$, we have the following lemma:

\begin{lemma} \label{lem:cGamma02}
    The functor \[p^{2,!} : \Vect \to \IndCoh(\fM_{0,2}^{\sst})\] admits a left adjoint $\Gamma_c(\fM_{0,2}^{\sst},-)$, and the trace map \[\Gamma_c(\fM_{0,2}^{\sst},\omega_{\fM_{0,2}^{\sst}}) \to k\] is an isomorphism.
\end{lemma}

\begin{proof}
    Following the description and notations in Section~\ref{sec:ss-modif}, we have an isomorphism \[\fM_{0,2}^{\sst} \simeq \colim_{\Delta^{\surj,\op}} \fA[\ell]~,\] where $\fA[\ell]$ is as in Definition~\ref{def:A[l]}, and with the transition map $f_{\alpha} : \fA[\ell'] \to \fA[\ell]$ corresponding to a surjection $\alpha : [\ell] \onto [\ell']$ given by compositions of the open inclusions \[\dotsb \overset{\GG_m}{\times} \dotsb \simeq \dotsb \overset{\GG_m}{\times} \GG_m \overset{\GG_m}{\times} \dotsb \into \dotsb \overset{\GG_m}{\times} \AA^1 \overset{\GG_m}{\times} \dotsb~.\] From this description, we get a presentation \[\IndCoh(\fM_{0,2}^{\sst}) \simeq \lim_{\Delta^{\surj}} \IndCoh(\fA[\ell])~,\] where the transition maps are given by $f_{\alpha}^!$. The functor $f_{\alpha}^!$ does not admit a left adjoint. However, its extension to pro-categories does have a left adjoint \[f_{\alpha,!}^{\operatorname{pro}} : \operatorname{Pro}(\IndCoh(\fA[\ell'])) \leftrightarrows \operatorname{Pro}(\IndCoh(\fA[\ell])) : f_{\alpha}^{\operatorname{pro},!}~.\] In particular, we have maps \begin{equation} \label{eq:unit_maps}
        p_{\ell,!}f_{\ell}^!\cF \simeq p_{\ell',!}f_{\alpha,!}^{\operatorname{pro}}f_{\alpha}^{\operatorname{pro},!}f_{\ell'}^!\cF \to p_{\ell',!}f_{\ell'}^!\cF~,
    \end{equation} where $p_{\ell} : \fA[\ell] \to \pt$ is the projection, and $f_{\ell} : \fA[\ell] \to \fM_{0,2}^{\sst}$ is the natural inclusion. To prove the existence of $p^2_!$, it is enough to prove the existence of $p_{\ell,!}$ for each $\ell \geq 0$: Indeed, assuming all $p_{\ell,!}$ exist, we can define \[p_{!}^2\cF = \colim p_{\ell,!}f_{\ell}^!\cF~,\] where the transition maps are given by~\eqref{eq:unit_maps}. We then have \begin{align*}
        \hom_{\Vect}(p_!^2\cF,V) & \simeq \lim \hom_{\Vect}(p_{\ell,!}f_{\ell}^!\cF,V) \simeq \lim \hom_{\fA[\ell]}(f_{\ell}^!\cF,p_{\ell}^!V) \\
        & \simeq \lim \hom_{\fA[\ell]}(f_{\ell}^!\cF, f_{\ell}^!p^!\cF) \simeq \hom_{\fM_{0,2}^{\sst}}(\cF, p^!V)~,
    \end{align*} where the last isomorphism follows from the description of hom spaces in limit categories.  

    To define $p_{\alpha,!}$, note that $p_{\ell}$ factors as the composition \[\GG_m \backslash \AA^1 \overset{\GG_m}{\times} \dotsb \overset{\GG_m}{\times} \AA^1 / \GG_m \to \GG_m \backslash \AA^1 \overset{\GG_m}{\times} \dotsb \overset{\GG_m}{\times} \pt \to \dotsb \to \GG_m \backslash \AA^1 / \GG_m \to B\GG_m \to \pt~,\] namely, $\ell$ contractions of $\AA^1/\GG_m$ to a point, followed by the projection $B\GG_m \to \pt$. Thus, we need to define compactly supported global sections for $\AA^1/\GG_m$ and $B\GG_m$. 
    
    For $\AA^1/\GG_m$, we have an equivalence $\IndCoh(\AA^1/\GG_m) \simeq \operatorname{Fil}(\Vect)$ between $\GG_m$-equivariant $k[t]$-modules and filtered vector spaces, such that the pullback functor $p^! : \Vect \to \IndCoh(\AA^1/\GG_m)$ corresponds to the non-positive constant filtration \[V \mapsto (\dotsb \to 0 \to V \xrightarrow{\id} V \xrightarrow{\id} \dotsb)~.\] This functor preserves both limits and colimits, and in particular admits a left adjoint. Concretely, we have \[\hom_{\Vect}(p_!\cF,k) \simeq \hom_{\AA^1/\GG_m}(\cF, \omega)~,\] i.e., \[p_!\cF \simeq \hom_{\AA^1/\GG_m}(\cF,p^!k)^{\lor}~.\]
    
    Taking $\cF = \omega$, we get $p_!\omega = \hom(\omega,\omega)^{\lor}$. A graded $k[t]$-morphism $\omega_{\AA^1} \to \omega_{\AA^1}$ is necessarily a multiple of the identity, and so \[p_!\omega_{\AA^1/\GG_m} \simeq k~.\]
    
    As for $\Gamma(B\GG_m,--)$, we use the description $\IndCoh(B\GG_m) \simeq \operatorname{Rep}(\GG_m) = \operatorname{gr}(\Vect)$ as the category of graded vector spaces. Under this equivalence, we have $\omega_{B\GG_m} \rightsquigarrow k[-1]$ --- the ground field in cohomological degree $1$ and graded degree $0$. We can then define $\Gamma_c(B\GG_m, -)$ to be the map \[V_{\bullet} \mapsto V_0[1]~.\] Indeed, for any $W \in \Vect$, we have \[\hom_{\operatorname{gr}(\Vect)}(V_{\bullet}, W \otimes k[-1]) \simeq \hom_{\Vect}(V_0, W[-1]) \simeq \hom_{\Vect}(V_0[1],W)~.\]

    Finally, to compute $\Gamma_c(\fM_{0,2}^{\sst},\omega)$, note that $\omega_{\AA^1/\GG_m}$ is given by the non-positive filtered object $\dotsb \to 0 \to k \xrightarrow{\id} k \to \dotsb$, and so $\Gamma_c(\AA^1/\GG_m,\omega) = k$, and similarly $\Gamma_c(B\GG_m, \omega_{B\GG_m}) \simeq k[-1][1] = k$. From that, we get $\Gamma_c(\AA^{\ell}/\GG_m^{\ell+1},\omega) \simeq k$, and taking colimit over the weakly-contractible diagram $\Delta^{\surj,\op}$, we get \[\Gamma_c(\fM_{0,2}^{\sst},\omega) \simeq \colim_{\Delta^{\surj,\op}} k \simeq k~.\]
\end{proof}

\begin{lemma} \label{lem:tr-p2}
    The functors \begin{align*}
        p^{2+2,!} : \Vect & \to \IndCoh(\fM_{0,2+2}^{\sst,\land}) \\
        p^{2+2,\nd,!} : \Vect & \to \IndCoh(\fM_{0,2+2}^{\sst,\land,\nd}) \\
        p^{1+2,!} : \Vect & \to \IndCoh(\fM_{0,1+2}^{\sst,\land})
    \end{align*} admit left adjoints $\Gamma_c(\fM_{0,2+2}^{\sst,\land},-)$, $\Gamma_c(\fM_{0,2+2}^{\sst,\land,\nd},-)$, and $\Gamma_c(\fM_{0,1+2}^{\sst,\land},-)$ respectively, and we have 
    \begin{align*}
        \Gamma_c(\fM_{0,2+2}^{\sst,\land},\omega) & \simeq k[\![t]\!]^{\lor} \\
        \Gamma_c(\fM_{0,2+2}^{\sst,\nd},\omega) & \simeq k \\
        \Gamma_c(\fM_{0,1+2}^{\sst,\land},\omega) & \simeq k~. 
    \end{align*}
\end{lemma}

\begin{proof}
    Using again the conventions of Section~\ref{sec:ss-modif}, we can write \begin{equation} \label{eq:M022-colim}
        \begin{split}
        \fM_{0,2+2}^{\sst,\land} & \simeq \colim_{\Delta^{\surj,\op}} \fA[\epsilon + \ell + \epsilon]~,\\
        \end{split}
    \end{equation} and the substack $\fM_{0,2+2}^{\sst,\nd}$ is given by the colimit over the substacks of $\fA[\epsilon + \ell + \epsilon]$ of tuples $(t_0,\dots,t_{\ell+1})$ where at least one $t_i$ is zero, while the substack $\fM_{0,1+2}^{\sst,\land}$ is a colimit over such tuples where $t_0 = t_{\infty} = 0$. The projections to a point factor as \begin{align*}
        \hat{\AA}^1 \overset{\GG_m}{\times} \AA^1 \overset{\GG_m}{\times} \dotsb \overset{\GG_m}{\times} \AA^1 \overset{\GG_m}{\times} \hat{\AA}^1 \to \GG_m \backslash \AA^1 \overset{\GG_m}{\times} \dotsb \overset{\GG_m}{\times} \AA^1 / \GG_m \to \pt ~.
    \end{align*} The first map is proper, so $*$-pushforward along it is left adjoint to $!$-pullback. Combining with Lemma~\ref{lem:cGamma02}, we get the existence of the compactly supported homology functors. 

    Next, note that \[\hom_{\Vect}(\Gamma_c(\fA[\epsilon + \ell + \epsilon], \omega_{\fA[\epsilon + \ell + \epsilon]}),k) \simeq \hom_{\fA[\epsilon + \ell + \epsilon]}(\omega_{\fA[\epsilon + \ell + \epsilon]},\omega_{\fA[\epsilon + \ell + \epsilon]})~,\] and the latter is given by $\ZZ_m^{\ell+1}$-graded $k[[t_0]\!][t_1,\dots,t_{\ell}][\![t_{\ell+1}]\!]$-linear morphisms \[k[\![t_0]\!][t_1,\dots,t_{\ell}][\![t_{\ell+1}]\!]dt_0\dotsb dt_{\ell+1} \to k[\![t_0]\!][t_1,\dots,t_{\ell}][\![t_{\ell+1}]\!]dt_0\dotsb dt_{\ell+1}~,\] with $t_0^{m_0}t_1^{m_1}\dotsb t_{\ell+1}^{m_{\ell+1}}$ in degree $(m_1 - m_0, \dots, m_{\ell+1} - m_{\ell})$. Such a morphism is determined by its value on $dt_0\dotsb dt_{\ell+1}$, whose image should be an element of degree $(0,\dots,0)$. Such an element is necessarily of the form $\alpha \cdot (t_0 \dotsb t_{\ell+1})^{m}dt_0\dotsb dt_{\ell+1}$ for $\alpha \in k$ and $m \geq 0$, and so \[\hom_{\Vect}(\Gamma_c(\fA[\epsilon + \ell + \epsilon], \omega_{\fA[\epsilon + \ell + \epsilon]}),k) \simeq k[\![t_0\dotsb t_{\ell}]\!]^{\lor}~.\] The transition maps in~\eqref{eq:M022-colim} corresponds to the map \[\alpha \cdot (t_0\dotsb t_{i-1}t_{i+1}\dotsb t_{\ell+1})^m \mapsto \alpha \cdot (t_0\dotsb t_{i-1}t_it_{i+1}\dotsb t_{\ell+1})^m\] so the diagram is constant, and we get \[\Gamma_c(\fM_{0,2+2}^{\sst,\land},\omega) \simeq k[\![t]\!]^{\lor}~.\]
    
    The restrictions to $\fM_{0,2+2}^{\sst,\nd}$ and $\fM_{0,1+2}^{\sst,\land}$ both correspond to the case $m=0$, and thus \[\Gamma_c(\fM_{0,1+2}^{\sst,\land},\omega) \simeq \Gamma_c(\fM_{0,2+2}^{\sst,\nd},\omega) \simeq k~.\]
\end{proof}

\begin{lemma} \label{lem:tr-q2}
	The trace maps for the projections \begin{align*}
		\bcM_{0,2+2,\Ran}^{\land} & \to \fM_{0,2+2}^{\sst,\land} \\
		\bcM_{0,2+1,\Ran}^{\land} & \to \fM_{0,2+1}^{\sst,\land} \\
		\bcM_{0,1+2,\Ran}^{\land} & \to \fM_{0,1+2}^{\sst,\land} \\
		\bcM_{0,2,\Ran} & \to \fM_{0,2}^{\sst} \\
	\end{align*} are isomorphisms.
\end{lemma}

\begin{proof}
	It suffices to prove for $q : \bcM_{0,2+2,\Ran}^{\land} \to \fM_{0,2+2}^{\sst,\land}$ --- the rest would follow by base-change. We can use the \'{e}tale charts \[\tau_{\ell} : \fA[\epsilon + \epsilon] \to \fM_{0,2+2}^{\sst,\land}\] (Definition~\ref{def:A[l]}) and prove by induction on $\ell$ that the trace map \begin{equation} \label{eq:trace-map}
		q_!^{\ell}\omega_{\bcM[\ell]} \to \omega_{\fA[\epsilon+\ell+\epsilon]}
	\end{equation} is an isomorphism for each $\ell$, where $q^{\ell} : \bcM[\ell] \to \fA[\epsilon + \ell + \epsilon]$ is the base-change of $q$ along $\tau_{\ell}$. For $\ell = 0$, we have \[\bcM[0] \simeq (\hat{\AA}^2 \times \Ran \GG_m) / \GG_m \to \hat{\AA}^2 / \GG_m\] and so the result follows from contractibility of $\Ran \GG_m$. Assuming the result for $0,\dots,\ell-1$, the map~\eqref{eq:trace-map} is an isomorphism away from $\overline{0}$ (since the complement of zero corresponds to less than $\ell$ nodes, thus covered by $\fA[\epsilon + \ell' + \epsilon]$ for some $\ell' < \ell$). Write $\imath : \hat{\AA}^1 \times \{\overline{0}\} \times \hat{\AA}^1 / \GG_m^{\ell+1} \into \fA[\epsilon + \ell + \epsilon]\}$ for the inclusion map. We have $\imath^* = \imath^![\ell]$. Let \[\begin{tikzcd}
	{\imath^*\bcM[\ell]} & {\bcM[\ell]} \\
	{(\hat{\AA}^1 \times \{\overline{0}\} \times \hat{\AA}^1) / \GG_m^{\ell+1}} & {\fA[\epsilon + \ell + \epsilon]}
	\arrow["{\tilde{\imath}}", from=1-1, to=1-2]
	\arrow["{\tilde{q}^{\ell}}", from=1-1, to=2-1]
	\arrow["{q^{\ell}}", from=1-2, to=2-2]
	\arrow["\imath", from=2-1, to=2-2]
	\end{tikzcd}\]
	be its pullback, so that $\imath^*\bcM[\ell] \simeq (\hat{\AA}^1 \times (\Ran \GG_m)^{\ell+1} \times \hat{\AA}^1) / \GG_m^{\ell+1}$. The map $\tilde{\imath}$ is representable and eventually coconnective, and so by \cite[Proposition~11.6.2]{Gai-ICS}, $\tilde{\imath}^*$ exists as well, and we have $\tilde{\imath}^*q^{\ell,!} \simeq \tilde{q}^{\ell,!}\imath^*$. We therefore have \[\tilde{\imath}^*\omega_{\bcM[\ell]} = \tilde{\imath}^*q^{\ell,!}\omega_{\fA[\epsilon + \ell + \epsilon]} \simeq \tilde{q}^{\ell,!}\imath^*\omega_{\fA[\epsilon + \ell + \epsilon]} \simeq \tilde{q}^{\ell,!}\imath^!\omega_{\fA[\epsilon + \ell + \epsilon]}[\ell] = \omega_{\imath^*\bcM[\ell]}[\ell]\] and therefore \begin{equation} \label{eq:imath-pullback}
		\imath^!q^{\ell}_!\omega_{\bcM[\ell]} \simeq \imath^*q^{\ell}_!\omega_{\bcM[\ell]}[-\ell] \simeq \tilde{q}^{\ell}_!\tilde{\imath}^*\omega_{\bcM[\ell]}[-\ell] \simeq \tilde{q}^{\ell}_!\omega_{\imath^*\bcM[\ell]}
	\end{equation} where the second isomorphism is the adjoint of the base change isomorphism $q^{\ell,!}\imath_* \simeq \tilde{\imath}_*\tilde{q}^{\ell,!}$. The RHS of~\eqref{eq:imath-pullback} maps isomorphically into $\omega_{(\hat{\AA}^1 \times \{\overline{0}\} \times \hat{\AA}^1) / \GG_m^{\ell+1}}$ by contractibility of $(\Ran \GG_m)^{\ell+1}$, and so $\imath^!$ applied to the cofiber of~\eqref{eq:trace-map} vanishes. However, since the cofiber is supported on the image of $\imath$ by hypothesis, the cofiber is zero by~\cite[Proposition~4.1.7]{Gai-ICS} (see also~\cite[Lemma~3.5.4]{DrGa})
\end{proof}

\begin{proof} [Proof of Proposition~\ref{prop:Z-unit}]
    Follows by combining Lemma~\ref{lem:tr-p2} and Lemma~\ref{lem:tr-q2}.
\end{proof}

\begin{corollary} \label{cor:int-unit}
    For an $n$-pointed stable curve $(X,x_1,\dots,x_n)$, the trace map \[\int_{X \backslash \{x_1,\dots,x_n\}} \pmb{1} \to k\] is an isomorphism.
\end{corollary}

\begin{proof}
    By Definition~\ref{def:fact-hom}, $\int_{X \backslash \{x_1,\dots,x_n\}} \pmb{1}$ is given by $!$-pushforward of the dualizing sheaf along the composition \[\bcM_{(X,x_1,\dots,x_n)} \xrightarrow{q} \Ran X \xrightarrow{p} \pt~,\] where $q$ is an isomorphism over smooth unmarked points, its fibers over nodes are isomorphic to $\bcM_{0,2+2,\Ran}^{\land,\nd}$, and its fibers over marked points are isomorphic to $\bcM_{0,1+2,\Ran}^{\land}$. In particular, by Proposition~\ref{prop:Z-unit}, the trace map $q_!\omega_{\bcM_{(X,x_1,\dots,x_n),\Ran}} \to \omega_{\Ran X}$ is an isomorphism. The result then follows by contractibility of $\Ran X$.
\end{proof}

\section{Computation of the zeroth homology} \label{sec:ss}

\subsection{Filtrations and construction of the spectral sequence} 

In this section, we define a filtration on the smoothing module $\tilde{\fZ}_{\cA}$ (Definition~\ref{def:smoothing-module}) coming from a combination of the cardinality filtration and the filtration by number of irreducible components on the Ran space of node-modifications $\bcM_{0,2+2,\Ran}^{\land}$ (Definition~\ref{def:M02plus1Ran}). This will result in  a spectral sequence computing the homology groups $H_*\tilde{\fZ}_{\cA} = H^{-*}\tilde{\fZ}_{\cA}$. By embedding the right and left Ran spaces $\bcM_{0,2+1,\Ran}^{\land}$ and $\bcM_{0,1+2,\Ran}^{\land}$ of puncture-modifications inside $\bcM_{0,2+2,\Ran}^{\land}$, we will obtain a description of the homology groups of the left and right puncture modules $\fZ_{\cA}^+$ and $\fZ_{\cA}^-$. By restricting to the strictly nodal locus $\bcM_{0,2+2,\Ran}^{\land,\nd}$ we will obtain a description of the homology of the node module $\fZ_{\cA}$, and by embedding $\bcM_{0,2,\Ran}$ inside $\bcM_{0,2+2,\Ran}^{\land}$ a description of the homology of the chiral Zhu algebra $\fZ_{\cA}^0$.

\begin{definition}
  For $p\geq 0$, let \[\jmath^{[\geq p]} : \bcM_{0,2+2,\Ran}^{[\geq p],\land} \into \bcM_{0,2+2,\Ran}^{\land}\] be the open subspace where the difference between numbers of weight-zero points and irreducible components is at least $p$, namely its $R$-points for a $k$-algebra $R$ are given by configurations $(X,x_0,x_1,x_{\infty-1},x_\infty,\{y_i^{\red}\}_{i \in I})$ such that $|I| \geq p + c - 2$, where $c$ is the number of irreducible components in $X_{R^{\red}}$. Let \[\hat{\imath}^{[p]} : \bcM_{0,2+2,\Ran}^{[p],\land} \into \bcM_{0,2+2,\Ran}^{[\geq p],\land}\] be the formal completion of $\bcM_{0,2+2,\Ran}^{[\geq p],\land} \backslash \bcM_{0,2+2,\Ran}^{[>p],\land}$ inside $\bcM_{0,2+2,\Ran}^{[\geq p]}$, namely configurations as above such that $|I| = p + c - 2$.
\end{definition}

Note that by stability, $\bcM_{0,2+2,\Ran}^{[\leq p],\land} = \emptyset$ for $p < 0$.

\begin{notation}
    For a tuple of positive integers $(n_1,\dots,n_{\ell})$, write \[S_{(n_1,\dots,n_{\ell})} = S_{n_1} \times \dotsb \times S_{n_{\ell}} \subset S_{n_1 + \dotsb + n_{\ell}}\] where $S_{n}$ denotes the symmetric group on a set of $n$ elements.
\end{notation}

\begin{definition}
    For a tuple of positive integers $(n_1,\dots,n_{\ell})$ with $\sum_{i=1}^{\ell} n_i = n$, let \[\bcM_{0,2,(n_1,\dots,n_{\ell})} \subset \bcM_{0,2,n}\] be the subscheme classifying configurations $(X,x_0,x_\infty,y_1,\dots,y_{n})$ such that $X$ has exactly $\ell$ irreducible components, and the points $y_{\sum_{i=1}^j n_i}, \dots, y_{\sum_{i=1}^{j+1}}$ are supported on the $j$-th irreducible component. Note that \begin{equation*}
        \bcM_{0,2,(n_1,\dots,n_{\ell})} \simeq \GG_m^{n_1} / \GG_m \times \dotsb \times \GG_m^{n_{\ell}} / \GG_m 
    \end{equation*} where the $i$-th component is the moduli of $n_i$ unmarked points on a smooth two-pointed genus zero curve.

    Let \[(\bcM_{0,2,(n_1,\dots,n_{\ell})})_{\dR/} = (\bcM_{0,2,(n_1,\dots,n_{\ell})})_{\dR} \times_{(\fM_{0,2}^{\sst})_{\dR}} \fM_{0,2}^{\sst}\] be the corresponding relative de Rham stack, and $(\bcM_{0,2+2,(n_1,\dots,n_{\ell})})_{\dR/}^{\land}$ the formal completion of $\bcM_{0,2+2,\Ran}$ along the closed embedding \[(\PP^1,0,1,\infty) \lor (-) \lor (\PP^1,0,1,\infty) : (\bcM_{0,2,(n_1,\dots,n_{\ell})})_{\dR/} \to \bcM_{0,2+2,\Ran}\]
\end{definition}

\begin{lemma} \label{lem:decomp}
    For each $p\geq 0$, there is a decomposition \begin{align*}
        \bcM_{0,2+2,\Ran}^{[p],\land} & \simeq \coprod_{\ell > 0} \coprod_{\substack{n_i > 0 \\ \sum n_i = p + \ell}} (\bcM^{\disj}_{0,2+2,(n_1,\dots,n_{\ell})})_{\dR/}^{\land} / S_{(n_1,\dots,n_{\ell})}
    \end{align*} where $(\bcM^{\disj}_{0,2+2,(n_1,\dots,n_{\ell})})_{\dR/}^{\land} \subset (\bcM_{0,2+2,(n_1,\dots,n_{\ell})})_{\dR/}^{\land}$ is the open substack where all marked points are disjoint, and the action of $S_{(n_1,\dots,n_{\ell})}$ is given by permuting the marked points inside each component. 
\end{lemma}

\begin{proof}
    Given an $R$-family $(X,x_0,x_1,x_{\infty-1},x_{\infty}, \{y_i^{\red}\}_{i \in I})$ of points in $\bcM_{0,2+2,\Ran}^{[p],\land}$, the number of disjoint weight-$\epsilon$ sections does not decrease in an open neighborhood of any $k$-point $\Spec k \to \Spec R$, while the number of irreducible components does not increase in an open neighborhood of any $k$-point. Thus, since their difference is constant and equal to $p$ by assumption, each of them is constant on connected components of $\Spec R$.
\end{proof}

We denote by \[c_{(n_1,\dots,n_{\ell})} : (\bcM_{0,2+2,(n_1,\dots,n_{\ell})}^{\disj})^{\land}_{\dR/} \to \bcM_{0,2+2,\Ran}^{[p],\land}\]
the quotient by $S_{(n_1,\dots,n_{\ell})}$ followed by the inclusion.

\begin{lemma} \label{lem:description_of_graded_pieces}
    For each tuple $(n_1,\dots,n_{\ell})$ of positive integers, the family \[(\bcM_{0,2+2,(n_1,\dots,n_{\ell})})_{\dR/}^{\land} \to \fM_{0,2+2}^{\sst}\] is isomorphic to the $\GG_m^{\ell}$-quotient of the trivial $\hat{\AA}^{\ell+1}$-family \[\hat{\AA}^{\ell+1} \times \prod_{i=1}^{\ell} \GG_{m,\dR}^{n_i} \to \hat{\AA}^{\ell+1}\] followed by the inclusion \[\hat{\AA}^{\ell+1}/\GG_m^{\ell} \to \fM_{0,2+2}^{\sst}\] The restriction to the disjoint locus is given by the $\GG_m^{\ell}$-quotient of the trivial $\hat{\AA}^{\ell+1}$-family\[\hat{\AA}^{\ell+1} \times \prod_{i=1}^{\ell} (\GG_m^{n_i} \backslash \Delta)_{\dR} \to (\bcM^{\disj}_{0,2+2,(n_1,\dots,n_{\ell})})_{\dR/}^{\land}\] Here $\hat{\AA}^{\ell+1}$ denotes the formal completion of $\AA^{\ell+1}$ at $0$, and $\Delta \subset \GG_m^{n_i}$ the union of all diagonal divisors.
\end{lemma}

\begin{proof}
    First note that, since de Rham stacks are not affected by formal completions, and over reduced points $(\bcM_{0,2+2,(n_1,\dots,n_{\ell})})_{\dR/}^{\land}$ contains curves with exactly $\ell+2$ irreducible components, we have \begin{align*}
        (\bcM_{0,2+2,(n_1,\dots,n_{\ell})})_{\dR/}^{\land} & \simeq (\bcM_{0,2,(n_1,\dots,n_{\ell})})_{\dR} \times_{(\fM_{0,2+2}^{\sst})_{\dR}} (\fM_{0,2+2}^{\sst,(\ell + 2)})^{\land} \\
        & \simeq (\GG_m^{n_1} / \GG_m \times \dotsb \times \GG_m^{n_{\ell}} / \GG_m)_{\dR} \times_{(\fM_{0,2+2}^{\sst})_{\dR}} (\fM_{0,2+2}^{\sst,(\ell + 2)})^{\land}
    \end{align*} where $(\fM_{0,2+2}^{\sst,(\ell+2)})^{\land}$ is the formal completion of $\fM_{0,2+2}^{\sst}$ along \[(\PP^1,0,1,\infty) \lor (-) \lor (\PP^1,0,1,\infty) : \fM_{0,2}^{\sst,(\ell)} \to \fM_{0,2+2}^{\sst}\] Here $B\GG_m^{\ell} \simeq \fM_{0,2}^{\sst,(\ell)} \subset \fM_{0,2}^{\sst}$ be the substack of exactly $\ell$ irreducible components. 
    
    By Proposition~\ref{prop:M0-ss-str}, the substack $\fM_{0,2}^{\sst,(\leq \ell + 2)}$ of at most $\ell + 2$ irreducible components admits an \'{e}tale cover by the stack $\AA^{\ell+1}/\GG_m^{\ell+2}$. In particular, \[(\fM_{0,2+2}^{\sst,(\ell+2)})^{\land} \simeq \hat{\AA}^{\ell+1}/\GG_m^{\ell}\] 
    and \[(\bcM_{0,2+2,(n_1,\dots,n_{\ell})})_{\dR/}^{\land} \simeq (\GG_m^{n_1} / \GG_m \times \dotsb \times \GG_m^{n_{\ell}} / \GG_m)_{\dR} \times_{(\AA^{\ell+1}/\GG^{\ell}_m)_{\dR}} \hat{\AA}^{\ell+1}/\GG^{\ell}_m\]
    
    Restricting to the disjoint locus then gives the desired result.
\end{proof}

\begin{definition}
    For a universal factorization algebra $\cA$, let $\VV \coloneqq \Vac(\cA)_{0,2+2}^{\epsilon}$ be the corresponding vacuum module over $\bcM_{0,2+2,\Ran}$. Let \[\VV^{[\geq p]} = \jmath^{[\geq p],!}\VV, \quad \VV^{[p]} = \hat{\imath}^{[p],!}\VV^{[\geq p]}, \quad \VV_{(n_1,\dots,n_{\ell})} = c_{(n_1,\dots,n_{\ell})}^!\VV^{[p]}\] Let $\cA_{\GG_m}$ be the factorization algebra over $\GG_m$ corresponding to $\cA$. For a finite set $I$, let $\cA_{\GG_m,I}$ be the corresponding D-module over $\GG_{m}^I$, and $\cA_{\GG_m,I}^{\disj}$ its restriction to $\GG_{m}^I \backslash \Delta$.
\end{definition}

\begin{proposition} \label{prop:ss}
    For a universal factorization algebra $\cA$, there is a spectral sequence \begin{align*}
        E^1_{p,q} = \bigoplus_{\ell > 0} \bigoplus_{\substack{n_i > 0 \\ \sum n_i = p + \ell}} H^{-\ell-p-q}\left(\omega_{\hat{\AA}^{\ell+1}} \bigotimes_{i=1}^{\ell} \Gamma_{\dR}(\GG_m^{n_i} \backslash \Delta, \cA^{\disj}_{\GG_m,n_i})\right)^{\GG_m^{\ell} \times S_{(n_1,\dots,n_{\ell})}} \Rightarrow & H_{p+q}\tilde{\fZ}_{\cA}~.
    \end{align*}
\end{proposition}

\begin{proof}
    From the filtration \[\bcM_{0,2+2,\Ran}^{[\geq 0], \land} \supset  \bcM_{0,2+2,\Ran}^{[\geq 1],\land} \supset \dotsb\] we get a filtration \[\VV = \VV^{[\geq 0]} \to \jmath^{[\geq 1]}_*\VV^{[\geq 1]} \to \jmath^{[\geq 2]}_*\VV^{[\geq 2]} \to \dotsb~.\]

    The fibers are given by \[\operatorname{fib}\left(\jmath^{[\geq p]}_*\VV^{[\geq p]} \to \jmath^{[\geq p+1]}_*\VV^{[\geq p+1]}\right) = \VV^{[p]}\] and so the associated spectral sequence is given by \[E^1_{p,q} = H^{-p-q}\Gamma\left(\bcM_{0,2+2,\Ran}^{[p],\land}, \VV^{[p]}\right)~.\]

    From Lemma~\ref{lem:decomp} we get a decomposition \[\Gamma\left(\bcM_{0,2+2,\Ran}^{[p],\land}, \VV^{[p]}\right) \simeq \bigoplus_{\ell > 0} \bigoplus_{\substack{n_i > 0 \\ \sum n_i = p + \ell}} \Gamma\left((\bcM_{0,2+2,(n_1,\dots,n_{\ell})}^{\disj})^{\land}_{\dR/}, \VV_{(n_1,\dots,n_{\ell})}\right)~.\] Under the isomorphism of Lemma~\ref{lem:description_of_graded_pieces}, each component can be described as \[\VV_{(n_1,\dots,n_{\ell})} \simeq \omega_{\hat{\AA}^{\ell+1}} \otimes \cA_{\GG_m,n_1}^{\disj} \otimes \dotsb \otimes \cA_{\GG_m,n_{\ell}}^{\disj}\] and so the result follows by taking homology, and using the isomorphism \[H^a(X/(\GG_m^{b} \times \Sigma_c), \cF) \simeq H^{a-b}(X,\tilde{\cF})^{\GG_m^b \times \Sigma_c}\] for any stack $X$ with an action of $\GG_m^b \times \Sigma_c$, and any $\cF \in \IndCoh(X/(\GG_m^{b} \times \Sigma_c))$ which corresponds to the $(\GG_m^{b} \times \Sigma_c)$-equivariant sheaf $\tilde{\cF} \in \IndCoh(X)^{\GG_m^{b} \times \Sigma_c}$ under the isomorphism \[\IndCoh(X/(\GG_m^{b} \times \Sigma_c)) \simeq \IndCoh(X)^{\GG_m^{b} \times \Sigma_c}~.\]
\end{proof}

\begin{lemma} \label{lem:first-quad-ss}
    Let $\cA = \Chev^{\ch}(\cB)$ for a coconnective universal chiral algebra $\cB$. Then  the spectral sequence of Proposition~\ref{prop:ss} satisfies $E^1_{p,q} = 0$ unless $p,q \geq 0$. In particular, \[H_0\tilde{\fZ}_{\cA} \simeq \operatorname{cof}\left(E^1_{1,0} \xrightarrow{d^1_{1,0}} E^1_{0,0}\right)~.\]
\end{lemma}

\begin{proof}
    Since $\GG_m^{n_i} \backslash \Delta$ is a smooth and affine variety of dimension $p+\ell$, and $\cA_{\GG_m,n_i}^{\disj} \simeq (\cB_1 \otimes \omega_{\GG_m})^{\boxtimes n_i} \mid \GG_m^{n_i} \backslash \Delta$ is concentrated in cohomological degree $\leq -n_i$, its de Rham cohomology is concentrated in degrees $\leq -n_i$. Altogether, we get that the complex $\omega_{\hat{\AA}^{\ell+1}} \bigotimes_{i=1}^{\ell} \Gamma_{\dR}(\GG_m^{n_i} \backslash \Delta, \cA^{\disj}_{\GG_m,I})$ is concentrated in degrees $\leq -\sum_{i=1}^{\ell} n_i = -(p + \ell)$. In particular, $H^{-\ell-p-q}(\omega_{\hat{\AA}^{\ell+1}} \bigotimes_{i=1}^{\ell} \Gamma_{\dR}(\GG_m^{n_i} \backslash \Delta, \cA^{\disj}_{\GG_m,I}))$ is zero unless $p,q \geq 0$.
\end{proof}

\subsection{The zeroth homology algebra}

In this section we fix a universal factorization algebra $\cA$ which is obtained from a quasi-conformal vertex algebra $V$, as in~\cite[Chapter~19]{FBZ}. Note that in such case, the D-modules $\cA^{\disj}_{\GG_m,n}$ are given by \[\cA^{\disj}_{\GG_m,n} \simeq (V \otimes \omega_{\GG_m})^{\boxtimes n} \mid \GG_m^n \backslash \Delta\] where the differential $\partial_x \in \operatorname{Diff}(\Spec k[x^{\pm 1}])$ acts by \[\partial_x \cdot (v \otimes x^ndx) = Tv \otimes x^ndx + v \otimes nx^{n-1}dx\] where $T : V \to V$ is the vertex algebra differential. In this case, we can describe the algebra $H_0\tilde{\fZ}_{\cA}$ explicitly.

\begin{definition}
    Let $U(V)$ be the universal enveloping algebra of $V$, as in Definition~\ref{def:UV}. Consider the vector space \[k[\![t_0,t_{\infty}]\!]^{\lor} \simeq k((t_0,t_{\infty}))dt_0dt_{\infty} / k[\![t_0,t_{\infty}]\!]dt_0dt_{\infty} \simeq \Gamma(\hat{\AA}^1 \times \hat{\AA}^1, \omega_{\hat{\AA}^1 \times \hat{\AA}^1})~.\] Denote by $\fU(V)$ the subalgebra of the trivial deformation $U(V) \otimes k[\![t_0,t_{\infty}]\!]^{\lor}$ given by \[\fU(V) = \left<\sigma \otimes t_0^{m_0}t_{\infty}^{m_{\infty}}dt_0dt_{\infty} \vcentcolon \deg(\sigma) = m_{\infty} - m_{0}\right> \subset U(V) \otimes k[\![t_0,t_{\infty}]\!]^{\lor}~,\] and the ideal \[\fI(V) = \left< \sigma \otimes t_0^{m_0}t_{\infty}^{m_{\infty}}dt_0dt_{\infty} \vcentcolon \sigma \in (U(V)_{<1-m_0} \cdot U(V))_{m_{\infty}-m_0} = (U(V) \cdot U(V)_{>1+m_{\infty}})_{m_{\infty} - m_0} \right>\]
\end{definition}

\begin{remark} \label{rem:subspaces}
    We will be interested in the quotient space \[\fU(V)/\fI(V)~.\]Restricting to the subspace $\{m_{\infty} = -1\} \subset \fU(V)/\fI(V)$, we get $m_0 = \deg(\sigma) + 1$, and so the subspace is isomorphic to the quotient $U(V)/U(V)\cdot U(V)_{>0}$. Similarly, the subspace $\{m_0 = -1\} \subset \fU(V)/\fI(V)$ is isomorphic to the quotient $U(V)/U(V)_{<0} \cdot U(V)$. The subspace $\{m_0 = m_{\infty} = -1\}$ is given by further restricting to elements with $\deg(\sigma) = 0$, and we obtain the Zhu algebra \[U(V)_0/(U(V) \cdot U(V)_{>0})_0 \simeq U(V)_0/(U(V)_{<0} \cdot U(V))_0 \simeq A(V)~.\]
\end{remark}

The main result of this section is the following:
\begin{theorem} \label{thm:H0ZA}
    For a universal factorization algebra $\fA$ associated to a quasi-conformal vertex algebra $V$, we have: 
    \begin{enumerate}
        \item $H_0\tilde{\fZ}_{\cA}$ is isomorphic to the quotient $\fU(V) / \fI(V)$.
        \item $H_0\fZ^0_{\cA}$ is isomorphic to the Zhu algebra $A(V)$ (Definition~\ref{def:Zhu}).
        \item $H_0\fZ_{\cA}^+$ is isomorphic to the left $U(V)$-module $U(V)/U(V) \cdot U(V)_{>0}  \underset{U(V)_0}{\otimes} A(V)$ associated to $A(V)$.
        \item $H_0\fZ_{\cA}^-$ is isomorphic to the right $U(V)$-module $A(V) \underset{U(V)_0}{\otimes} U(V)/U(V)_{<0}\cdot U(V)$ associated to $A(V)$.
        \item  $H_0\fZ_{\cA} \simeq H_0\fZ^+_{\cA} \underset{H_0\fZ_{\cA}^0}{\otimes} H_0\fZ_{\cA}^-$ is isomorphic to the vertex bimodule underlying the mode-transition algebra $\fA(V)$ (Definition~\ref{def:MTA}).
        \item The isomorphism in (2) is an isomorphism of associative algebras.
    \end{enumerate}
\end{theorem}

We will prove that using the spectral sequence of Proposition~\ref{prop:ss}. In order to do that, we will first have to understand the differential \[d_{1,0}^1 : E^1_{1,0} \to E^1_{0,0}~.\]

\begin{notation}
    For $\ell > 0$ and $1 \leq i \leq \ell$, define $\pmb{n}^{\ell,i} \in \ZZ^{\ell}$ by \[\pmb{n}^{\ell,i}_j = \begin{cases} 1 & i \neq j \\ 2 & i = j \end{cases}~.\] Denote by \begin{align*}
        (d_{1,0}^1)^{\ell,i}_{\ell'}  : & H^{-\ell-1}\left(\omega_{\hat{\AA}^{\ell+1}} \bigotimes_{j=1}^{\ell} \Gamma_{\dR}\left(\GG_m^{\pmb{n}^{\ell,i}_j} \backslash \Delta, (V \otimes \omega_{\GG_m})^{\boxtimes \pmb{n}^{\ell,i}_j}\right)\right)^{\GG_m^{\ell} \times S_{\pmb{n}^{\ell,i}}} \\
        \to & H^{-\ell'}\left(\omega_{\hat{\AA}^{\ell'+1}} \bigotimes_{i=1}^{\ell'} \Gamma_{\dR}(\GG_m, V \otimes \omega_{\GG_m})\right)^{\GG_m^{\ell'}}
    \end{align*} the $((\ell,i),\ell')$-th matrix element of the differential $d^1_{1,0}$.
\end{notation}

\begin{definition}
    Let $\pi : W[1] \to \AA^1$ be the smooth family of Definition~\ref{def:univ-punct-sm-curve}. Let $\cA_{W[1]/\AA^1}$ be the $\AA^1$-factorization algebra over $W[1]$ corresponding to $\cA$, and \[\cA_{W[1]/\AA^1,2} = f_2^{!}\cA_{W[1]/\AA^1}\] its restriction along \[\begin{tikzcd}
	{(W[1] \times_{\AA^1}W[1])_{\AA^1\mh\dR}} & {} & {\Ran(W[1]/\AA^1)} \\
	& {\AA^1}
	\arrow["{f_2}", from=1-1, to=1-3]
	\arrow["{\pi_2}"', from=1-1, to=2-2]
	\arrow["{\pi_{\Ran}}", from=1-3, to=2-2]
\end{tikzcd}\]  
\end{definition}

Note that the restriction $W[1]_{\GG_m} \to \GG_m$ can be identified with the trivial family $\GG_m \times \GG_m \to \GG_m$ using either the $x$-coordinate or the $y$-coordinate. The restriction $W[1] \times_{\AA^1} \{0\}$ can be identified with the curve $\GG_m \sqcup \GG_m$ using the coordinate $x$ for the first copy and $y$ for the second. In particular, we have \[(W[1] \times_{\AA^1}W[1])_{\AA^1\mh\dR} \times_{\AA^1} \GG_m \simeq (\GG_m \times \GG_m)_{\dR} \times \GG_m \to \GG_m\] and \[(W[1] \times_{\AA^1}W[1])_{\AA^1\mh\dR} \times_{\AA^1} \hat{\AA}^1 \simeq (\GG_m \sqcup \GG_m)_{\dR} \times \hat{\AA}^1\] We have a map \begin{equation} \label{eq:restr-diag}
    (\cA_{W[1]_{\GG_m}/\GG_m,2}) \mid (\GG_m^2 \backslash \Delta)_{\dR} \times \GG_m \to (\cA_{W[1]_{\GG_m}/\GG_m,2})[1] \mid \Delta_{\dR} \times \GG_m
\end{equation} given by residue at the diagonal, and a map \begin{equation} \label{eq:residue-at-zero}
    \cA_{W[1]_{\GG_m}/\GG_m,2} \to \cA_{W[1]_{\hat{\AA}^1}/\hat{\AA}^1,2}[1]
\end{equation} given by the residue at $0$. From the factorization isomorphism we get an isomorphism \[(\cA_{W[1]_{\GG_m}/\GG_m,2}) \mid (\GG_m^2 \backslash \Delta)_{\dR} \times \GG_m \simeq (\cA_{W[1]_{\GG_m}/\GG_m,1})^{\otimes^!_{\GG_m} 2} \mid (\GG_m^2 \backslash \Delta)_{\dR} \times \GG_m\] where $\cA_{W[1]/\AA^1,1}$ is the restriction of $\cA_{W[1]/\AA^1}$ along $W[1]_{\AA^1\mh\dR} \to \Ran(W[1]/\AA^1)$. Since $\cA$ is obtained from a vertex algebra $V$, we can write $\cA_{W[1]_{\GG_m}/\GG_m,1} \simeq V \otimes \omega_{\GG_m^2}$. Thus, \eqref{eq:restr-diag} defines a family of chiral multiplication maps \begin{equation} \label{eq:ch-mult}
    \omega_{\GG_m} \otimes (V \otimes \omega_{\GG_m})^{\otimes_{\GG_m}^! 2} \mid (\GG_m \backslash \Delta)_{\dR} \times \GG_m \to \omega_{\GG_m} \otimes (V \otimes \omega_{\GG_m})^{\otimes_{\GG_m}^! 2}[1] \mid \Delta_{\dR} \times \GG_m 
\end{equation} while \eqref{eq:residue-at-zero} defines by restriction a map \begin{equation} \label{eq:residue-at-zero-disj}
    \omega_{\GG_m} \otimes (V \otimes \omega_{\GG_m})^{\otimes_{\GG_m}^! 2} \mid (\GG_m \backslash \Delta)_{\dR} \times \GG_m \to \omega_{\hat{\AA}^1} \otimes (V \otimes \omega_{\GG_m})[1]
\end{equation}

\begin{definition}
    Let \begin{align*}
        \fL_0(V) & = H^{-1}\Gamma_{\dR}(\GG_m, V \otimes \omega_{\GG_m}) \\
        \fL_1(V) & = H^{-2}\Gamma_{\dR}(\GG_m^2 \backslash \Delta, (V \otimes \omega_{\GG_m})^{\boxtimes 2})^{S_2}~.
    \end{align*} 
    
    Let \[\mu : \omega_{\GG_m} \otimes \fL_1(V) \to \omega_{\GG_m} \otimes \fL_0(V)\] be the map induced by~\eqref{eq:ch-mult}, and \[\nu : \omega_{\GG_m} \otimes \fL_{1}(V) \to \omega_{\hat{\AA}^1} \otimes \fL_{0}(V) \otimes \fL_{0}(V)\] the map induced by~\eqref{eq:residue-at-zero-disj}. Denote by \[\mu^{\GG_m} : \fL_1(V) \to \fL_0(V)\] and \[\nu^{\GG_m} : \fL_1(V) \to (\omega_{\hat{\AA}^1} \otimes \fL_0(V) \otimes \fL_0(V))^{\GG_m}\] the $\GG_m$-fixed points of the above maps, with $\GG_m$ acting diagonally on the tensor products.
\end{definition}

Explicitly, if we choose coordinates $(z,w,z',w',t)$ for the family $\WW[1] \times_{\AA^1} \WW[1] \to \AA^1$, then elements of $\omega_{\GG_m} \otimes H^{-2}\Gamma_{\dR}(\GG_m^2 \backslash \Delta, (V \otimes \omega_{\GG_m})^{\boxtimes 2})$ are spanned by sections of the form \[\xi = t^mdt \otimes [u \boxtimes v \otimes z^az'^b(z-z')^cdzdz'];\quad u,v \in V,\; a,b,c,m \in \ZZ\] (here $[-]$ denotes a class in de Rham cohomology), and elements of $\omega_{\GG_m} \otimes \fL_1(V)$ are given by $S_2$-equivariant sections, which are spanned by sections of the form $\xi + \sigma\xi$ with $\xi$ as above, where $\sigma$ is the generator of the $S_2$ action. Elements of $\omega_{\GG_m} \otimes \fL_0(V)$ are spanned by sections of the form \[t^mdt \otimes [v \otimes z'^ndz'];\quad v \in V, \; m,n \in \ZZ~,\] and elements of $\omega_{\hat{\AA}^1} \otimes \fL_0(V) \otimes \fL_0(V)$ are spanned by sections of the form \[t^mdt \otimes [u \otimes z^{n}dz] \otimes [u \otimes w'^{-n'}dw']; \quad u,v \in V,\; m < 0,\; n,n' \in \ZZ~.\] The map $\mu$ can be then expressed as \begin{equation} \label{eq:mu}
    \begin{split}
        \mu : \xi + \sigma \xi \mapsto & t^mdt \otimes \operatorname{Res}_{z-z'} \sum_{n \in \ZZ} u_{(n)}v \otimes z^az'^b(z-z')^{c-n-1}d(z-z')dz' \\
        = & t^mdt \otimes  \operatorname{Res}_{z-z'} \sum_{n \in \ZZ} u_{(n)}v \otimes ((z-z')+z')^az'^b(z-z')^{c-n-1}d(z-z')dz' \\
        = & t^mdt \otimes \operatorname{Res}_{z-z'} \sum_{n \in \ZZ} u_{(n)}v \otimes \sum_{j \geq 0} \binom{a}{j} z'^{a+b-j}(z-z')^{c+j-n-1}d(z-z')dz' \\
        = & t^mdt \otimes \sum_{j \geq 0} u_{(c+j)}v \otimes \binom{a}{j} z'^{a+b-j}dz' \\
        = & t^{m + a + b + c - \deg(u) - \deg(v)}dt \otimes \sum_{j \geq 0} u_{(c+j)}v \otimes \binom{a}{j} w'^{-(a+b-j)}dw'^{-1}
    \end{split}
\end{equation}

The map $\nu$ can be expressed as \begin{equation} \label{eq:nu}
    \begin{split}
        \nu : \xi + \sigma \xi = & t^{m}dt \otimes [u \boxtimes v \otimes z^az'^{b}(z-z')^cdzdz'] \\
        - & t^{m}dt \otimes [v \boxtimes u \otimes z^bz'^{a}(-1)^c(z-z')^cdzdz'] \\
        = & t^{m + b + 1 - \deg(v)}dt \otimes [u \boxtimes v \otimes z^aw'^{-b}(z-tw'^{-1})^cdzdw'^{-1}] \\
        - & t^{m + a + 1 - \deg(u)}dt \otimes [v \boxtimes u \otimes z^bw'^{-a}(-1)^c(z-tw'^{-1})^cdzdw'^{-1}] \\
        \mapsto & t^{m + b + 1 - \deg(v)}dt \otimes [u \boxtimes v \otimes z^aw'^{-b}(z-tw'^{-1})^cdzdw'^{-1}] \mod \text{regular in }t\\
        - & t^{m + a + 1- \deg(u)}dt \otimes [v \boxtimes u \otimes z^bw'^{-a}(-1)^c(z-tw'^{-1})^cdzdw'^{-1}] \mod \text{regular in }t\\
        = & \sum_{j = 0}^{\deg(v)-m-b-1} (-1)^j\binom{c}{j} t^{m+b+j-\deg(v)}dt \otimes [u \boxtimes v \otimes z^{a+c-j}w'^{-(b+j)}dzdw'^{-1}] \\
        - & \sum_{j=0}^{\deg(u)-m-a-1} (-1)^{j+c}\binom{c}{j} t^{m+a+j-\deg(u)}dt \otimes [v \boxtimes u \otimes z^{b+c-j}w'^{-(a+j)}dzdw'^{-1}]
    \end{split}
\end{equation}

\begin{remark}
    $\fL_0(V)$ coincides with the underlying vector space of the Lie algebra $\fL(V)$ defined in~\ref{def:UV}, where the element $[v \otimes x^ndx] \in \fL_0(V)$ corresponds to the element $v_{[n]} \in \fL(V)$.
\end{remark}

\begin{lemma} \label{lem:str_E1}
    For $\ell > 0$ and $1 \leq i \leq \ell$, the $(\ell,i)$-th term of $E_{1,0}^1$ is given by \[\left(\omega_{\hat{\AA}^{\ell+1}} \otimes \fL_0(V)^{\otimes i-1} \otimes \fL_1(V) \otimes \fL_{0}(V)^{\otimes \ell-i}\right)^{\GG_m^{\ell}}~.\] For $\ell' > 0$, the $\ell'$-th term of $E^1_{0,0}$ is given by \[\left(\omega_{\hat{\AA}^{\ell'+1}} \otimes \fL_0(V)^{\otimes \ell'}\right)^{\GG_m^{\ell'}}\] For $\ell' \notin \{\ell, \ell + 1\}$, $(d_{1,0}^1)^{\ell,i}_{\ell'} = 0$. The map $(d_{1,0}^1)^{\ell,i}_{\ell}$ is given by $\mu^{\GG_m}$ applied to the $i$-th component, and the map $(d_{1,0}^1)^{\ell,i}_{\ell+1}$ is given by $\nu^{\GG_m}$ applied to the $i$-th component.
\end{lemma}

\begin{proof}
    The description of the $(i,\ell)$-th and $\ell'$-th components follows directly from the definition of $\fL_0(V), \fL_1(V)$ and the fact that we are taking the highest degree cohomology. The description of the boundary maps follows from the fact that the closure of $(\bcM_{0,2+2,(\pmb{n}^{\ell,i})})^{\land}_{\dR/}$ inside $\bcM_{0,2+2,\Ran}^{[\leq 1],\land}$ contains two components --- the first is $(\bcM_{0,2+2,\pmb{1}_{\ell}})_{\dR/}^{\land}$, corresponding to the collision of the two marked points on the $i$-th component, and the second is $(\bcM_{0,2+2,\pmb{1}_{\ell+1}})^{\land}_{\dR/}$, corresponding to the two degenerations of the $i$-th component into two components, where the second degeneration is obtained from the first by swapping the two marked points.
\end{proof}

\begin{proof} [Proof of Theorem~\ref{thm:H0ZA}]
    First we prove (1). By Lemma~\ref{lem:first-quad-ss}, $H_0\tilde{\fZ}_{V}$ is the cofiber of $d^1_{1,0}$. By Lemma~\ref{lem:str_E1}, $E_{0,0}^1$ is spanned by elements of the form \begin{equation} \label{eq:zeta}
        \zeta = t_0^{m_0}dt_0 \cdot \sigma_1 \cdot t_1^{m_1}dt_1 \dotsb \sigma_{\ell} \cdot t_{\ell}^{m_{\ell}}dt_{\ell}; \quad \sigma_i = [v^i \otimes z_i^{n_i}dz_i] \in \fL_0(V), m_i < 0
    \end{equation} for some $\ell>0$, where $z_i$ is a coordinate on the $i$-th irreducible component, such that the weight of the $i$-th $\GG_m$ action is zero for $1 \leq i \leq \ell$: \[\deg_i(\zeta) = m_{i-1} + 1 + \deg(\sigma_i) - (m_{i} + 1) = 0~.\] By induction, we get \begin{equation} \label{eq:partial_deg}
        0 \geq m_{i} + 1 = \sum_{j=1}^{i}\deg(\sigma_j) + m_0 + 1
    \end{equation} and in particular \begin{equation} \label{eq:total_deg}
        \sum_{j=1}^{\ell}\deg(\sigma_j) = m_{\ell} - m_0~.
    \end{equation}

    Again by Lemma~\ref{lem:first-quad-ss}, the image of $d_{1,0}^1$ is spanned by $\mu(\xi) - \nu(\xi)$ for $\xi$ a section of $\omega_{\hat{\AA}^{\ell^-+\ell^++2}} \otimes \fL_0(V)^{\otimes \ell^-} \otimes \fL_1(V) \otimes \fL_0(V)^{\otimes \ell^+}$ of the form \[\xi = \zeta^0 \cdot \tau \cdot \zeta^{\infty}\] with \[\zeta^{s} = t_0^{m_0^{s}}dt_0 \cdot \sigma^{s}_1 \cdot t_1^{m^{s}_1}dt_1 \dotsb \sigma_{\ell^{s}}^{s} \cdot t_{\ell^{s}}^{m_{\ell^{s}}^{s}}dt_{\ell^{s}};\quad s \in \{0,\infty\}\] as in \eqref{eq:zeta}, and with \[\tau = [u \boxtimes v \otimes z^{a}z'^{b}(z-z')^{c}dzdz']\] for $u,v \in V$, $a,b,c,m \in \ZZ$. Namely, $\xi$ represents a section over the locus with $\ell^{0} + \ell^{\infty} + 1$ irreducible components, with the $\ell^0+1$-th component supporting two weight-$\epsilon$ points with coordinates $z,z'$, and any other component supporting a single point. We require that $\xi$ will be $\GG_m^{\ell^{0} + \ell^{\infty} + 1}$-equivariant. Equivariance with respect to the first $\ell^{0}$ and last $\ell^{\infty}$ copies is as in~\eqref{eq:partial_deg} and~\eqref{eq:total_deg}. For the action of the $\ell^{0}+1$-th copy, equivariance implies \[\deg_{\ell^{0}+1}(\xi) = m_{\ell^{0}} + \deg(\tau) - m_{0}^{\infty} = 0~.\] Such $\xi$ can be identified with the $\GG_m^{\ell^0+\ell^{\infty}+2}$-equivariant section of $\omega_{\hat{\AA}^{\ell^{0}+1} \times \GG_m \times \hat{\AA}^{\ell^{\infty}+1}} \otimes \fL_0(V)^{\otimes \ell^0} \otimes \fL_1(V) \otimes \fL_0(V)^{\otimes \ell^{\infty}}$ given by \[\zeta^0 \cdot (t^{m}dt \cdot \tau) \cdot \zeta^{\infty};\quad m = m_0^{\infty} = m^0_{\ell^0} + \deg(\tau)~.\] In particular \[m_{\ell^{\infty}}^{\infty} - m_0^{0} = \sum_{i = 1}^{\ell^{0}} \deg(\sigma_i^{0}) + \deg(\tau) + \sum_{i=1}^{\ell^{\infty}} \deg(\sigma^{\infty}_i)\]
    
    Using the description in~\eqref{eq:mu} and~\eqref{eq:nu}, and the identification $z^{\infty}_{1} = w'^{-1} = t^{-1}z'$, we get the relations: \begin{equation} \label{eq:relations-H0}
        \begin{split}
            & \zeta^0 \cdot \sum_{j = 0}^{\deg(v)-m-b-1} (-1)^j\binom{c}{j} t^{m+b+j-\deg(v)}dt \cdot [u \boxtimes v \otimes z^{a+c-j}(z_1^{\infty})^{b+j}dzdz_1^+] \\
            - & \sum_{j=0}^{\deg(u)-m-a-1} (-1)^{j+c}\binom{c}{j} t^{m+a+j-\deg(u)}dt \otimes [v \boxtimes u \otimes z^{b+c-j}(z_1^{\infty})^{a+j}dzdz_1^{\infty}] \cdot \zeta^{\infty} \\
            = & \zeta^0 \cdot \left(t^{m - \deg(\tau)}dt \cdot \sum_{j \geq 0} \otimes u_{(c+j)}v \binom{a}{j} (z_1^{\infty})^{a+b-j}dz^{\infty}_1\right) \cdot \zeta^{\infty}~.\\
        \end{split}
    \end{equation}
    Define \[f : E_{0,0}^1 \to \fU(V) / \fI(V)\] by \[f : t_0^{m_0}dt_0 \cdot [v^1 \otimes z_1^{n_1}dz_1] \cdot t_1^{n_1}dt_1 \dotsb [v^{\ell} \otimes z_{\ell}^{n_{\ell}}dz_{\ell}] \cdot t^{m_{\ell}}_{\ell}dt_{\ell} \mapsto [t_0^{m_0}dt_0 \cdot v^1_{[n_1]} \dotsb v^{\ell}_{[n_{\ell}]}]\cdot t_{\infty}^{m_{\ell}}dt_{\infty}~.\] By~\eqref{eq:partial_deg}, any left factor $t_0^{m_0}dt_0 \cdot [v^1 \otimes x_1^{n_1}] \dotsb [v^i \otimes x_{i}^{n_i}dx_i] \cdot t^{m_i}_idt_i$ has degree $m_i - m_0 < -m_0$, and by~\eqref{eq:total_deg} the total degree is $m_{\infty} - m_0$, so $f$ is a bijection on generators. The relations on the target are given by the truncated Jacobi relations \begin{equation}
        \begin{split}
            t_0^{m_0}t_{\infty}^{m_{\infty}}dt_0dt_{\infty} \otimes & \sum_{j = 0}^{\deg(v^{i+1}_{[b]}v^{i+2}_{[n_{i+2}]}\dotsb v^{\ell+1}_{[n_{\ell+1}]})-m_{\infty}-1} (-1)^j\binom{c}{j} v^1_{[n_1]}\dotsb v^{i}_{[a+c-j]}v^{i+1}_{[b+j]}\dotsb v^{\ell+1}_{[n_{\ell+1}]} \\
            - & \sum_{j=0}^{\deg(v^{i}_{[a]}v^{i+2}_{[n_{i+2}]}\dotsb v^{\ell+1}_{[n_{\ell+1}]})-m_{\infty}-1}(-1)^{j+c}\binom{c}{j}v^1_{[n_1]}\dotsb v^{i+1}_{[b+c-j]}v^{i}_{[a+j]} \dotsb v^{\ell+1}_{[n_{\ell+1}]} \\
            = & \sum_{j \geq 0}\binom{a}{j}v^1_{[n_1]} \dotsb (v^i_{(c+j)}v^{i+1})_{[a+b-j]} \dotsb v^{\ell+1}_{[n_{\ell+1}]}
        \end{split}
    \end{equation} for any $1 \leq i \leq \ell$. Since we have \[m_{\ell+1} - m_{i} = \deg(v^{i+1}_{[n_{i+1}]}\dotsb v^{\ell+1}_{[n_{\ell+1}]})~,\] this is precisely the image under $f$ of the relations~\eqref{eq:relations-H0}, and we get an induced isomorphism \[\overline{f} : E_{0,0}^{1}/E_{1,0}^1 \iso \fU(V)/\fI(V)~.\]

    To prove (5), note that $H_0\fZ_{\cA} \subset H_0\tilde{\fZ}_{\cA}$ is given by restriction to the locus of at least one node, namely the product $t_0 \dotsb t_{\ell} = 0$ is zero. A section $\zeta$ is supported on this subspace if $m_i = -1$ for some $0 \leq i \leq \ell$. In particular, $\zeta$ can be written as a product $\zeta = \zeta^0 \cdot t_i^{-1}dt_i \cdot \zeta^{\infty}$, where \begin{align*}
        \zeta^0 & = t_0^{m_0}dt_0 \cdot \sigma_1 \dotsb t_{i-1}^{m_{i-1}}dt_{i-1} \cdot \sigma_{i} \\
        \zeta^{\infty} & = \sigma_{i+1} \cdot t_{i+1}^{m_{i+1}}dt_{i+1} \dotsb \sigma_{\ell} \cdot t^{m_{\ell}}dt_{\ell}~.
    \end{align*} The section $\zeta^0 \cdot t_{i}^{-1}dt_i$ is an element of the subspace of $\fU(V)/\fI(V)$ where $m_{\infty} = -1$, while the section $t_i^{-1}dt_i \cdot \zeta^{\infty}$ is an element of the subspace $\{m_0=-1\}$. Thus, the result follows from Remark~\ref{rem:subspaces}.

    For (2), (3), and (4), note that the subspace $H_0\fZ_{\cA}^0$ (resp. $H_0\fZ_{\cA}^+$, resp. $H_0\fZ_{\cA}^-$) corresponds to the subspace where first and last component (resp. last, resp. first) components are separated by a node, namely $m_0 = m_{\infty} = -1$ (resp. $m_{\infty} = -1$, resp. $m_0 = -1$), so the result follows again from Remark~\ref{rem:subspaces}.

    Finally, to prove (6), note that the associative algebra structure on $H_0\fZ^0_{V}$ is given by concatenation of components, namely \begin{align*}
        & (t_0^{-1}dt_0 \cdot \sigma_1 \cdot t_1^{m_1}dt_1 \dotsb \sigma_{\ell} \cdot t_{\ell}^{-1}dt_{\ell}) \cdot (t_0^{-1}dt_0 \cdot \tau_1 \cdot t_1^{n_1}dt_1 \dotsb \cdot \tau_{r} \cdot t^{n_r}_rdt_r) \\
        = & t_0^{-1}dt_0 \cdot \sigma_1 \cdot t_1^{m_1}dt_1 \dotsb \sigma_{\ell} \cdot t_{\ell}^{-1}dt_{\ell} \cdot \tau_1 \cdot t_{\ell+1}^{n_1}dt_{\ell+1} \dotsb \cdot \tau_{r} \cdot t^{n_r}_{\ell+r}dt_{\ell+r}~. 
    \end{align*}
    Since in this case the data of $m_{i}$ is redundant (as it is determined by the degrees of $\sigma_j$), this corresponds exactly to the concatenation product on $A(V)$, given by \[\sigma_1 \dotsb \sigma_{\ell} \cdot \tau_1 \dotsb \tau_{r}~.\]
\end{proof}

\printbibliography

\end{document}